\documentclass[12pt]{amsart}
\usepackage[utf8]{inputenc}
\usepackage[margin=1in]{geometry}
\usepackage[greek,english]{babel}
\usepackage{amsmath}
\usepackage{amsfonts}
\usepackage{amsthm}
\usepackage{amssymb}
\usepackage{mathtools}
\usepackage{tikz-cd}
\usepackage{enumerate}
\usepackage{xcolor}
\usepackage[toc,page]{appendix}

\usepackage{enumitem}
\usepackage[pagebackref,colorlinks,linkcolor=red,citecolor=blue,urlcolor=blue,hypertexnames=true]{hyperref}
\usepackage[pagebackref,hypertexnames=true]{hyperref}
\DeclareMathOperator{\rank}{rank}
\DeclareMathOperator{\Inf}{Inf}

\numberwithin{equation}{section}
\newcommand{\U}{\mathcal{Q}}
\newcommand{\R}{\mathbb{R}}
\newcommand{\Q}{\mathbb{Q}}
\newcommand{\Z}{\mathbb{Z}}

\newcommand{\C}{\mathbb{C}}

\newcommand{\N}{\mathbb{N}}
\newcommand{\K}{\mathcal{K}}

\newcommand{\T}{\mathcal{T}}
\newcommand{\TT}{\mathbb{T}}
\newcommand{\CC}{\mathcal{C}}
\newcommand{\W}{\mathcal{W}}
\newcommand{\B}{\mathcal{B}}
\newcommand{\J}{\mathcal{J}}

\makeatletter
\newcommand\nnfootnote[1]{%
  \begin{NoHyper}
  \renewcommand\thefootnote{}\footnote{#1}%
  \addtocounter{footnote}{-1}%
  \end{NoHyper}
}

\makeatother

\newtheorem{theorem}{Theorem}[section]

\newtheorem{definition}[theorem]{Definition}
\newtheorem{lemma}[theorem]{Lemma}
\newtheorem{proposition}[theorem]{Proposition}
\newtheorem{remark}[theorem]{Remark}
\newtheorem{example}[theorem]{Example}
\newtheorem{corollary}[theorem]{Corollary}
\newtheorem{conjecture}[theorem]{Conjecture}

\title[Extensions of quasidiagonal $C^*$-algebras]{Extensions of quasidiagonal $C^*$-algebras and controlling the $K_0$-map of embeddings}
\author{Iason Moutzouris}

\begin{document}
\maketitle
\nnfootnote{2020 Mathematics Subject Classification. Primary: 46L05.\par
}

\begin{abstract}
    We study the validity of the Blackadar-Kirchberg conjecture for extensions of separable, nuclear, quasidiagonal $C^*$-algebras that satisfy the UCT. More specifically, we show that the conjecture for the extension has an affirmative answer if the ideal lies in a class of $C^*$-algebras that is closed under local approximations and contains all separable ASH algebras, as well as certain classes of simple, unital $C^*$-algebras and crossed products of unital $C^*$-algebras with $\Z$.
\end{abstract}
\makeatletter
\@setabstract
\makeatother

\section{Introduction}
We know that every quasidiagonal $C^*$-algebra is stably finite \cite[Prop 7.1.15]{brown2008textrm}. The converse is not true, even for exact $C^*$-algebras. A counterexample is $C^*_r(\mathbb{F}_2)$. Indeed, it is stably finite because it has a faithful trace, but it is not quasidiagonal because $\mathbb{F}_2$ is not amenable \cite[Cor 7.1.17]{brown2008textrm}. However, it is a very interesting question to ask under which extra conditions the converse holds. Because $C^*_r(\mathbb{F}_2)$ is not nuclear \cite[Thm 2.6.8]{brown2008textrm}, we should ask for a nuclearity assumption. In \cite{Blackadar1997GeneralizedIL}, Blackadar and Kirchberg conjectured that it is enough to assume nuclearity and separability.
\begin{conjecture}[Question 7.3.1, \cite{Blackadar1997GeneralizedIL}]\label{Blackadar-Kirchberg} If $A$ is separable, stably finite and nuclear, then it is quasidiagonal.
\end{conjecture}
Although the the conjecture is still open, there are some partial results confirming it. For instance, the conjecture holds when $A$ (in addition to the conjecture assumption) has either one of the following properties: i) $A$ is  simple and satisfies the UCT  (\cite [Cor. B]{tikuisis2017quasidiagonality}),ii) $A$ is traceless \cite[Cor C]{gabeembeddings} or iii) $A=B\rtimes_{\sigma}\Z$, where $B$ is an AH algebra of real rank zero \cite[Thm 1.1]{MR3915315}. \par
Because nuclearity and separability are closed under taking extensions, a positive answer on the Blackadar-Kirchberg conjecture would automatically guarantee a positive answer to the following conjecture.

\begin{conjecture}\label{extension_conjecture}
If $$\begin{tikzcd}
0 \arrow[r] & I \arrow[r, "\iota"] & A \arrow[r, "\pi"] & B  \arrow[r] & 0
\end{tikzcd}$$
is a short exact sequence, with $I,B$ separable, nuclear quasidiagonal, then $A$ is quasidiagonal iff $A$ is stably finite. 
\end{conjecture}

It has to be mentioned that under the aforementioned assumptions, $A$ is not automatically quasidiagonal. One easy counterexample is \par
$$\begin{tikzcd}
0 \arrow[r] & \K \arrow[r, "\iota"] & \T \arrow[r, "\pi"] & C(\TT)  \arrow[r] & 0
\end{tikzcd}$$
where $\T$ denotes the Toeplitz algebra and $\K$ the compact operators. Notice that $\K$ and $C(\TT)$ are quasidiagonal, but $\T$ is not. However, it is not stably finite either, as it contains a non-unitary isometry (the unilateral shift). Trying to verify Conjecture \ref{extension_conjecture} is the starting point for this paper. \par
In \cite{brown2004extensions}, Brown and Dadarlat, connected Conjecture \ref{extension_conjecture} with the presence of a property related to K-theory which they named \emph{$K_0$-embedding property} (see Def \ref{K_0_emb}). More specifically, if the ideal has the $K_0$-embedding property and the quotient is separable, nuclear, quasidiagonal and satisfies the UCT, then the Blackadar-Kirchberg Conjecture holds for the $C^*$-algebra in the middle (see also Remark \ref{drop_compacts}). If $A$ is a separable and quasidiagonal $C^*$-algebra, the presence of $K_0$-embedding property means that for every $G\leq K_0(A)$ with $G\cap K_0(A)^+=\{0\}$, there exists an embedding $\rho:A\hookrightarrow B$, where $B$ is quasidiagonal and $\rho_*(G)=0$. Note that $G\cap K_0(A)^+=\{0\}$ is easily seen to be a necessary condition for the existence of such an embedding. Apart from specific easy cases (see end of section 2), not much has been known regarding which $C^*$-algebras have the $K_0$-embedding property. \par
Let $\mathcal{Y}$ be the class of $C^*$-algebras that can be written as a finite direct sum of algebras belonging to either one of the following classes:
\begin{enumerate}[label=\roman*.]
     \item $D\rtimes_{\sigma} \Z$, where $D$ is separable, nuclear, unital, quasidiagonal and satisfies the UCT, $\sigma:\Z\rightarrow Aut(D)$ is a minimal action and $D$ has a $\sigma$-invariant trace.
    \item separable ASH algebras.

\end{enumerate}
In our paper, we will show that all algebras in $\mathcal{Y}$ have the $K_0$-embedding property. Combining this with other results from our paper (mainly Proposition \ref{UHF-stability} and Proposition \ref{local_approximations}) and the aforementioned comments, we deduce the following Theorem.
\begin{theorem}\label{main-theorem}
Assume that
$$\begin{tikzcd}
0 \arrow[r] & A \arrow[r, "\iota"] & E \arrow[r, "\pi"] & B  \arrow[r] & 0
\end{tikzcd}$$
is a short exact sequence where $A$ is separable, $A\otimes \U$ is locally approximated by algebras in $\mathcal{Y}$ and $B$ is separable, nuclear, quasidiagonal, satisfying the UCT. \par
Then $E$ is quasidiagonal iff it is stably finite.
\end{theorem}
\begin{remark}
Note that if $A$ is separable, simple, nuclear, unital, quasidiagonal and satisfies the UCT, then $A\otimes \U$ has finite nuclear dimension by the main result of \cite{MR4228503}, so by \cite[Thm. 6.2(iii)]{tikuisis2017quasidiagonality}, it is an ASH algebra. Thus $A$ satisfies the hypothesis of Theorem \ref{main-theorem}.
\end{remark}
\begin{remark}
Note that $\mathcal{Y}$ is closed under taking matrix algebras, direct sums, as well as tensoring with $\mathcal{Q}$. So, if a $C^*$-algebra $A$ satisfies the hypothesis of Theorem \ref{main-theorem}, then $A\otimes \K$ (and more generally $A\otimes D$ for every AF-algebra $D$) also satisfies the hypothesis of Theorem \ref{main-theorem}.
\end{remark}
Let $E$ be as in Theorem \ref{main-theorem}. Then it is not simple and usually it does not admit any faithful trace. For example, the latter is guaranteed if the ideal $A$ is stable. So, Theorem \ref{main-theorem} verifies the Blackadar and Kirchberg conjecture for a large class of $C^*$-algebras that have no faithful trace. Actually, even the case where the $C^*$-algebra arises as an extension of a separable, nuclear, quasidiagonal $C^*$-algebra with the UCT by $C(X)\otimes \K$, cannot be deduced in a straightforward way from any of the previous results that we could find in the literature.

\par

For the proof of Theorem \ref{main-theorem}, we need results regarding ordered K-theory from \cite{Goodearl1986PartiallyOA}, as well as techniques from \cite{SPIELBERG1988325}, in order to construct many of the $K_0$ maps. We then rely on a classification Theorem of Schafhauser \cite[Cor. 5.4]{Schafhauser2018SubalgebrasOS} to "lift" these maps to the $C^*$-algebra level, and thus achieve the existence of the embeddings needed to show the presence of $K_0$-embedding property. For the case when the ideal is a separable ASH algebra, we use results and techniques from \cite{elliott2017classification}, \cite{elliott2020decomposition} and \cite{phillips2007recursive}.

The paper is structured as follows: In Section 2, we give the definitions, some needed preliminaries and mention already known examples of $C^*$-algebras with the $K_0$-embedding property. In Section 3, we show that $K_0$-embedding Property is a local property. In Section 4, we set the stage for the last two sections, in which we produce new examples of $C^*$-algebras with the $K_0$-embedding Property. More specifically, in Section 5 we show the $K_0$-embedding property for direct sums of certain $C^*$-algebras, including simple ones and crossed products with $\Z$ via minimal actions (Proposition \ref{direct_sums}). It has to be noted that the results we show in Section 4 make us realize for which direct sums we can show the $K_0$-embedding property when applying our strategy. In Section 6, we establish $K_0$-embedding property for separable ASH algebras. There is also an appendix, on which, for the sake of completion, we present a proof of a result regarding crossed products, that is essential for the proof of Proposition \ref{crossed_products}.\par
Throughout the paper $\N^*=\{1,2,3,...\}, \U$ will be the universal UHF algebra, by $F\subset \subset A$ we will denote a finite subset of A, while $\otimes$ will be the minimal tensor product. Moreover, for an element $a$ in a $C^*$-algebra, $\bar a$ will denote an image under some quotient map. For a $C^*$-algebra $A$, $sr(A)$ and $RR(A)$ will be the stable and real rank of a $C^*$-algebra respectively. If $F$ is a set and $C$ is a $C^*$-algebra, with $F\subset_{\varepsilon} C$ we mean that for every element in $a\in F$ there is an element $b\in C$ that is $\varepsilon$-close to $a$ in norm.We will use the abbreviation ccp for completely positive and contractive maps. And finally, if $\tau\in T(A)$ is a trace and $a\in M_n(A)$, we will sometimes abuse the notation and write $\tau(a)=Tr\otimes \tau(a)=\sum_{i=1}^n \tau (a_{ii})$.

\section{Preliminaries and basic examples}
We start by giving a few definitions: \par
\begin{definition}[Def. 7.1.1, \cite{brown2008textrm}] A separable $C^*$-algebra is \textbf{quasidiagonal} if there exists a sequence of asymptotically multiplicative and asymptotically isometric ccp maps $\phi_n:A\rightarrow M_{k(n)}$.
\end{definition}

\begin{definition}[Def. 4.4, \cite{brown2004extensions}]\label{K_0_emb} We say that a separable and quasidiagonal $C^*$ algebra A has the \textbf{$K_{0}$-embedding property} if for every $ G\leq K_{0}(A)$ such that $G\cap K_{0}^{+}(A)=\{0\}$, there is a faithful *-homomorphism $\rho : A\rightarrow C$, where C is quasidiagonal and $\rho_{*}(G)=0$.
\end{definition}

\begin{definition}[Def. 4.3, \cite{brown2004extensions}] We say that a separable and quasidiagonal $C^{*}$ algebra $A$ has the \textbf{QD-extension property} if for every separable, nuclear, quasidiagonal $C^{*}$-algebra $B$ which satisfies the UCT, and for every short exact sequence
$$\begin{tikzcd}
0 \arrow[r] & A\otimes \K \arrow[r, "\iota"] & E \arrow[r, "\pi"] & B  \arrow[r] & 0
\end{tikzcd}$$
$E$ is quasidiagonal iff $E$ is stably finite.
\end{definition}

We need to add more notation:
\begin{definition} Let $(G,G^+)$ be an ordered abelian group. A subgroup $H\leq G$ is called \textbf{singular} if $H \cap G^+=\{0\}$. If $x\in G$, we will say that $x$ is \textbf{singular} if $\Z x\cap G^+=\{0\}$.\par
\end{definition}

\begin{remark}\label{drop_compacts}
 If $A$ is separable and quasidiagonal, by \cite[Prop 4.6]{brown2004extensions} it has the $K_0$-embedding property iff it has the qd extension property. Assume that $A$ is separable, quasidiagonal, has the $K_0$-embedding property and
 $$\begin{tikzcd}
0 \arrow[r] & A \arrow[r, "\iota"] & E \arrow[r, "\pi"] & B  \arrow[r] & 0
\end{tikzcd}$$
is a short exact sequence with $B$ separable, nuclear, quasidiagonal, satisfying the UCT and $E$ is stably finite. Then, if we tensor everything with $\K$, the sequence remains exact. Also $E\otimes \K$ is stably finite and the properties of $B$ pass to $B\otimes \K$. Because $A$ has the $K_0$-embedding property and hence the QD-extension property, $E\otimes \K$ is quasidiagonal hence $E$ is quasidiagonal. Thus, by the aforementioned and \cite[Thm. 4.11]{brown2004extensions}, in order to prove Conjecture \ref{extension_conjecture} (when $B$ satisfies the UCT), it is enough to show that every nuclear, separable, quasidiagonal $C^*$-algebra has the $K_0$-embedding property.
 \end{remark}

Let $A$ be a $C^*$-algebra. Then we can write $A\otimes \U$ as an inductive limit, i.e.
$$ A\otimes \U=\varinjlim A\otimes M_{n!}(\C).$$
where the connecting maps are
$$id\otimes \phi_i, \text{ where } \hspace{2mm} \phi_i:M_{i!}(\C)\rightarrow M_{(i+1)!}(\C) \text{ is defined by } \phi_i(a)=diag(a,a,...,a).$$ and $id\otimes \mu_i$, where $$\mu_i: M_{i!}(\C)\rightarrow \mathcal{Q}$$ is the inclusion from the definition of $(\U, \mu_i)=\varinjlim M_{n!}(\C)$ as an inductive limit. \par
By the stability of $K_0$ (\cite[Proposition 4.3.8]{Rrdam2000AnIT}) $K_0(A)\cong K_0(M_n(A))$. \par
By the K\"{u}nneth Theorem \cite[Thm. 23.1.2]{blackadar1998k} (or the continuity of $K_0$ \cite[Thm. 6.3.2]{Rrdam2000AnIT}), we have
\begin{equation}
    \label{2.1}
    K_0(A\otimes \U)=K_0(A)\otimes \Q.
\end{equation}
Let $x,y \in K_0(A)$ and $a,b,c,d\in \Z$. Then $$x\otimes \frac{a}{b}+ y\otimes \frac{c}{d}=x\otimes \frac{ad}{bd}+y\otimes \frac{bc}{bd}=$$ $$=adx\otimes \frac{1}{bd}+bcy\otimes \frac{1}{bd}=(adx+bcy)\otimes \frac{1}{bd}.$$ So, every (finite) sum of elementary tensors in $K_0(A)\otimes \Q$ is still an elementary tensor. Thus
$$K_0(A)\otimes \Q=\{x\otimes y \hspace{2mm} \vert \hspace{2mm} x\in K_0(A),y\in \Q \}.$$
Recall that for a $C^*$-algebra $A$, the \textbf{positive cone} of $K_0(A)$ is
$$K_0(A)^+:=\{[p]_0 \hspace{3mm} | \hspace{3mm} p\in P_{\infty}(A)\}.$$
We can put an order on $K_0(A)$ as follows:
$$x\leq y \text{ iff } y-x\in K_0(A)^+.$$
The following proposition is well-known but we present a proof for the sake of completion.
\begin{proposition}\label{k_0 of U-stable} The positive cone of $K_0(A\otimes \U)$ is $$ K_0(A\otimes \U)^+=\{x\otimes y \hspace{2mm} \vert \hspace{2mm} x\in K_0(A)^+,y\in \Q_{\geq 0}\}.$$

\begin{proof}
First of all, note that by the continuity of $K_0$ we have $$K_0(A\otimes \U)^+=\bigcup_{i=1}^{\infty}(id\otimes \mu_i)_*(K_0(A\otimes M_{n!}(\C))^+).$$
Note that $(id\otimes \mu_i)_*: K_0(A)\rightarrow K_0(A)\otimes \Q$ is the division with $i$. To see this, observe that  $(id\otimes \phi_i)_*: K_0(A)\rightarrow K_0(A)$ is the multiplication with $i+1$ and then see that the properties of the inductive limit are satisfied. Result follows.
\end{proof}
\end{proposition}

The aforementioned will help us show our first permanence property.
\begin{proposition}\label{UHF-stability}
 If A is separable and quasidiagonal then $A\otimes \U$ has the $K_0$-embedding property, iff $A$ has it.
 \end{proposition}
\begin{proof}
Let $\iota:A\hookrightarrow A\otimes \U$ be the natural embedding. Assume first that $A\otimes \U$ has the $K_0$-embedding Property.
Let $G\leq K_0(A)$ singular. By Proposition \ref{k_0 of U-stable}, $\iota_*(G)$ is singular. Hence, we have an embedding $\rho :A\otimes\U \hookrightarrow D$, where $D$ is quasidiagonal and $\rho_*(\iota_*(G))=0$.
After considering $\rho\circ\iota: A\hookrightarrow D$, we deduce that $A$ has the desired property. Conversely, assume that $A$ has the $K_0$-embedding property and let $G\leq K_0(A\otimes \U)$ singular. Set  $H:=\{a\in K_0(A)\hspace{2mm}| \hspace{2mm} \exists b\in \Q \text{ such that }a\otimes b \in G\}$. Note that $H$ is a singular subgroup of $K_0(A)$. By assumption there exists a quasidiagonal $C^*$-algebra $B$ and an embedding $h:A\hookrightarrow B$, such that $h_*(H)=0$. Consider the map $h \otimes id: A\otimes \U \hookrightarrow B\otimes \U$. Note that $(h \otimes id)_*(G)=0$, so $A\otimes \U$ has the $K_0$-embedding property.
 \end{proof}
 If we want to show that the $K_0$-embedding property is satisfied for a class of $C^*$-algebras, the previous proposition allows us (at least in the most cases) to assume in addition that the $C^*$-algebras are $\U$-stable.
This helps a lot because for every  separable, stably finite, unital, $\U$-stable $C^*$-algebra, the ordered $K_0$ group is well behaved. \begin{definition}[Def. 7.2.5, \cite{Rrdam2000AnIT}] \label{unperforated_def}
An ordered abelian group $(G,G^+)$ is called \textbf{unperforated} if every $x\in G$ for which $nx\geq 0$ for some $n\in \N$ satisfies $x\geq 0$.
\end{definition} 
Notice that an unperforated ordered group must be torsion free. If $A$ is $\U$-stable, then $K_0(A) \cong K_0(A)\otimes \Q$ by \ref{2.1} and $K_0(A)$ is unperforated. Indeed $A\cong A\otimes \U$, so if $nx\geq 0$, then $nx=a\otimes b$, where $a\in K_0(A)^+$ and $b\in \R_{\geq 0}$ by Proposition \ref{k_0 of U-stable}. Thus $x=a\otimes \frac{b}{n}\in K_0(A)\otimes \Q$. Hence $a\geq 0$ in $K_0(A)$. \par

For the rest of the section, all groups will be abelian, unless clearly stated otherwise.
Let $(G,G^+,u)$ be a scaled, ordered, countable group that satisfies $G\cong G\otimes \Q.$ In this case the isomorphism is on the category of ordered groups and $(G\otimes \Q)^+=\{a\otimes b\hspace{2mm} \vert \hspace{2mm} a\in G^+,b\in \Q_{\geq 0}$\}. Note that $(G\otimes \Q,(G\otimes \Q)^+)$ is indeed an ordered group by \cite [Prop. 2.3]{MR845669}. It follows that it is unperforated and hence torsion free. Let $H\leq G$ be a singular subgroup. Set
$$\mathcal{F}=\{ L\leq\ G \hspace{2mm} \vert \hspace{2mm}  L\cap G^+={0}, H\subset L \}. $$  Then $\mathcal{F}\neq \emptyset \hspace{2mm}(H\in \mathcal{F})$
and for every $(L_i)_{i\in I}$ increasing chain in $\mathcal{F}, \bigcup_{i\in I} L_{i} \in \mathcal{F}$.
Hence, Zorn's Lemma applies and $\mathcal{F}$ has a maximal element. Such subgroup will be called \emph{maximally singular}. Observe that every singular subgroup is contained inside a maximally singular subgroup.\vspace{2mm}

Let $H_0\leq G$ be a singular subgroup. Consider the following property for $H_0$:
\begin{equation}
    \label{2.2}
    \text{If there exists }k\in \N^* \text{ such that } kx\in H_0 \text{ then it follows that } x\in H_0.
\end{equation}
\begin{lemma}\label{(2.2)-holds}
If $H_0$ is maximally singular or $H_0\cong H_0\otimes \Q$, then (2.2) holds.
\begin{proof}
Assume first that $H_0$ is maximally singular. For the sake of contradiction, assume that there exist $x\in G$ and $k\in \N^*$, such that $ kx\in H_0$, but $x\notin H_0$. Consider $H_0^{'}:=span_{\Z}\{H_0,x\}\leq G$. Then $H_0^{'}\supsetneq H_0$ and also $H_0^{'}$ is singular because $kH_0^{'}\subset H_0$, contradicting the maximality of $H_0$.
Assume now that $H_0\cong H_0\otimes \Q$ and let $x\in G$ such that $kx\in H_0$. By assumption, $kx=a\otimes b$ for some $a\in H_0$ and $b\in \Q$. It follows that $x=a\otimes \frac{b}{k}\in H_0\otimes \Q\cong H_0$.
\end{proof}
\end{lemma}

Let $G_0=G\slash H_0$. We can put an order on $G_0$ as follows:
\begin{equation}
    \label{2.3}
    G_0^+=\{\bar x\in G_0 \hspace{2mm} \vert \hspace{2mm} \exists y\in H_0 \text { such that } x+y\geq 0 \text{ in } G\}.
\end{equation}

\begin{lemma}\label{order_on_quotient}
$(G_0,G_0^+,\bar u)$, as defined in (2.3) is a scaled ordered group and the quotient map $\pi:G\rightarrow G_0$ is positive. If, in addition (2.2) holds, then it is unperforated. Moreover, if $H_0$ is maximally singular, then $(G_0,G_0^+,\bar u)$ is also totally ordered, hence dimension group.
\begin{proof}
This is essentially contained in the proof of \cite[Lemma 1.14]{SPIELBERG1988325}. However, for the sake of completion, we will repeat the arguments here: \par
First of all, it is straightforward to check that $\pi(x)\in G_0^+$ for every $x\in G^+.$ Moreover, we have: \par
\textbf{Claim 1:} $G_0^+ + G_0^+\subset G_0^+.$ \par
\textbf{Proof:} Let $\bar x,\bar y \in G_0^+$. By definition there exist $x_1,y_1 \in H_0$ such that $x+x_1\geq 0$ and $y+y_1\geq 0$ in $G$. Adding the two inequalities together, we get $(x+y)+(x_1+y_1)\geq 0$. Notice that $x_1+y_1 \in H_0$, so $x+y\in G_0^+$ by definition. \qed \vspace{4mm}

\textbf{Claim 2:} $G_0^+-G_0^+=G_0$. \par
\textbf{Proof:} This follows from the facts that $G^+-G^+=G$ and $\pi(G^+)\subset G_0^+$. \qed \vspace{4mm}

\textbf{Claim 3:} $G_0^+\cap -G_0^+=\{0\}$. \par
\textbf{Proof:} Let $x\in G$ such that $\bar x$ and $-\bar x \in G_0^+$. Then there exists $e,f\in H_0$ such that $x+e\geq 0$ and $-x+f\geq 0$. If $x+e>0$, then adding the two relations together yields $e+f>0$. But $e+f\in H_0$, so we have a contradiction. Thus $x+e=0$, which implies $x\in H_0$, so $\bar x=0$. \qed. \vspace{4mm}

Combining Claims 1,2 and 3, we deduce that $(G_0,G_0^+)$ is an ordered group. By the first sentence of the proof, $\pi$ is positive. Let $\bar x\in G_0$. Because $u$ is an order unit for $(G,G^+)$, there exist $n\in \N^*$ such that $-nu \leq x \leq nu $. Because $\pi$ is positive, $-n \bar u\leq \bar x \leq n \bar u$. Because $\bar x$ is arbitrary, $\bar u$ is an order unit for  $(G_0,G_0^+)$. \par
Assume now that (2.2) holds. We will show that the ordered group is unperforated. More specifically, let $n\bar x\in G_0^+$ for some positive integer $n$ and $\bar x \in G_0$. Then there is $y\in H_0$ such that $nx+y\geq 0$ in $G$. Observe that because $G\cong G\otimes \Q$, $\frac{y}{n}$ is a well-defined element of $G$. By (2.2) $\frac{y}{n} \in H_0$ and because $G$ is unperforated, $x+\frac{y}{n}\in G^+$. So, $\bar x \in G_0^+,$ thus $(G_0,G_0^+)$ is unperforated. \par
Assume in addition that $H_0$ is maximally singular. We will show that $(G_0,G_0^+)$ is totally ordered. Let $\bar x \in G_0\backslash \{0\}$. Notice that $x\notin H_0$ so the fact that $H_0$ is maximally singular yields that there exist $n\in \Z$ and $e\in H_0$ such that $nx+e>0$. If $n>0$, then $n\bar x>0$. Because $(G_0,G_0^+)$ is unperforated, $\bar x >0$. Similarly, if $n<0$, then $\bar x<0$. Hence $(G_0,G_0^+)$ is totally ordered.

\end{proof}
\end{lemma}

Notice that if (2.2) does not hold, then $G_0$ will have torsion and thus perforation. \par
Recall that a \emph{state} on an ordered group $(G,G^+,u)$ is a group homomorphism $\phi:G\rightarrow \R$ such that $\phi(u)=1$ and $\phi(G^+)\subset \R^{+}$ . The set of states is defined to be $S(G)$. We call a state $\phi$ \emph{faithful} if $\phi(x)>0 \hspace{2mm}\forall x\in G^+\backslash \{0\}$. Every (non-zero) scaled ordered group has a state by \cite[Cor. 4.4]{Goodearl1986PartiallyOA}.\par
Following the previous notation, we can see that if $\phi \in S(G) \text{ with } \phi(H_0)=0$, then
$\widetilde{\phi}: G_0\rightarrow \R$ with $\widetilde{\phi}(\bar x)=\phi(x)$ is a state on $(G_0,G_0^+,\bar u)$.
And moreover, the map
\begin{equation}
    \label{2.4}
    \Phi:\{\phi\in S(G) \hspace{2mm} \vert \hspace{2mm}\phi(H_0)=0\}\rightarrow S(G_0)
\end{equation}
with $\Phi(\phi)=\widetilde{\phi}$ is a bijection.\par
If $\tau\in T(A)$, then we have an induced state $\hat{\tau}\in S(K_0(A))$ such that $\forall p\in P_n(A)$,\par
$\hat{\tau}[p]_0=Tr\otimes \tau(p)=\sum_{i=1}^n \tau(p_{ii}).$\par
Thus, we have an induced map
\begin{equation}
    \label{2.5}
    \Psi:T(A){\rightarrow S(K_0(A))}
\end{equation}
with $\Psi(\tau)=\hat{\tau}$. \par

This map is onto if $A$ is exact by \cite[Thn. 5.11]{haagerup2014quasitraces} and \cite[Thm. 3.3]{BLACKADAR1992240}. If we also have $RR(A)=0$, \cite[Prop. 1.1.12]{rordam2002classification} and the comments below it, yield that the map is a bijection. Note that if $A$ is an AF-algebra, $RR(A)=0$.\par
 Moreover, for a scaled ordered group $(G,G^+,u)$, the group of infinitesimals is
 \begin{equation}
     \label{2.6}
     \Inf(G)=\{x\in G : \hspace{2mm}\rho(x)=0 \hspace{2mm} \forall \rho \in S(G)\}.
 \end{equation}
 If, moreover, the ordered group is unperforated, it is known that
 \begin{equation}
     \label{2.7}
     \Inf(G)=\{x\in G: \hspace{2mm} u+nx\geq 0 \hspace{2mm} \forall n\in \Z\}
 \end{equation}
 (see for instance \cite[Lemma 2.5]{dadarlat2004morphisms} if the ordered group is simple but not necessarily unperforated. In our case the proof is identical.) \par
 \vspace{4mm}
 We will need the following simple Lemma.
 \begin{lemma}\label{infinitesimals_after_homomorphisms}
 Let $A,B$ be unital and stably finite $C^*$-algebras and $\phi:A\rightarrow B$ a unital *-homomorphism. Then $\phi_*(\Inf(K_0(A)))\subset \Inf(K_0(B))$.
 \begin{proof}
 Let $x\in \Inf(K_0(A))$ and $\rho_B \in S(K_0(B),K_0(B)^+,[1_B]_0)$. Then $\rho_A:=\rho_B \circ \phi_*: K_0(A)\rightarrow \R$ is a state in $(K_0(A), K_0(A)^+, [1_A]_0)$. Note that we use the fact that $\phi(1_A)=1_B$. By (\ref{2.6}), $\rho_A(x)=0$. Hence, $\rho_B(\phi_*(x))=0$. Because $\rho_B$ is arbitrary, $\phi_*(x)\in \Inf(K_0(B))$ by (\ref{2.6}).
 \end{proof}
 \end{lemma}

 The next Lemma will be very useful in Section 5.
 \begin{lemma}\label{order_lemma}
 Let $(G,G^+,u)$ be a scaled ordered, countable abelian group with $G\cong G\otimes \Q$ via an order isomorphism. Let also $H_1\subset H_2\subset G$ such that $H_1\leq G$ is a subgroup with $H_1\cap G^+=\{0\}$ and $H_2\subset G$ is a subsemigroup with $H_2\cap -G^+=\{0\}$ and $H_2\cap -H_2=H_1$. Then there exists a state $\rho\in S(G)$ such that $\rho(H_1)=0$ and $\rho(x)\geq 0$ for every $x\in H_2$.
 \begin{proof}
 Set $G_0=G/ H_1$ and let $\pi:G\rightarrow G_0$ be the quotient map. By Lemma \ref{order_on_quotient}, $(G_0,G_0^+,\bar u)$ endowed with the order as defined in (2.3), is a scaled ordered group. Set $P=G_0^++\overline{H_2}$. We will show that $(G_0,P,\bar u)$ is a scaled ordered group. Indeed, $P-P=G_0$, because $P\supset G_0^+$. Because $H_2$, and hence $\overline{H_2}$, is a semigroup, we have that $P+P\subset P$. Let $x\in P\cap -P$. Then $x=a_1+b_1=-a_2-b_2$, where $a_i\in G_0^+$ and $b_i\in \overline{H_2}.$ Hence $$(a_1+a_2)+(b_1+b_2)=0.$$ But $a_1+a_2\in G_0^+$, while $b_1+b_2\in \overline{H_2}$. By assumption, we have that $H_2\cap -G^+=\{0\}.$ Thus $-G_0^+\cap \overline{H_2}=\{0\}$, which yields
 \begin{equation}
     \label{2.8}
     a_1+a_2=0 \text{ and } b_1+b_2=0.
 \end{equation}
 Moreover, by assumption, we have that $H_2 \cap -H_2=H_1$, which implies that $\overline{H_2} \cap -\overline{H_2}=\{0\}$. Also, $G_0^+ \cap -G_0^+=\{0\}$. So, (2.8) yields that $a_1=a_2=b_1=b_2=0$, thus $x=0$. This means that $P\cap -P=\{0\}$. It follows that $(G_0,P)$ is an ordered group. Because $\bar u$ is an order unit on $(G_0,G_0^+)$ and $P\supset G_0^+$, $\bar u$ must be an order unit on $(G_0,P)$. By \cite[Cor. 4.4]{Goodearl1986PartiallyOA}, $(G_0,P,\bar u)$, has a state, call it $\tau$. Because $P\supset G_0^+$, $\tau \in S(G_0,G_0^+)$ and $\tau(z)\geq 0$, for every $z\in \overline{H_2}$. Finally, because the map in (2.4) is onto, $\tau=\Phi(\rho)$, for some $\rho\in S(G)$. It is not difficult to show that $\rho$ satisfies the desired properties.
 \end{proof}
 \end{lemma}

 \begin{remark}\label{trace_kills_singular}
 Let $A$ be a separable, exact, stably finite, unital and $\U$-stable $C^*$-algebra and $G\leq K_0(A)$ a singular subgroup. Note that the comments after Definition \ref{unperforated_def} guarantee that $K_0(A)$ satisfies the hypothesis of Lemma \ref{order_lemma}. By Lemma \ref{order_lemma} for $H_1=H_2=G$ and the fact that the map in (2.5) is onto, we deduce that there exists $\tau \in T(A)$ such that $\hat{\tau}(G)=0$. If, moreover, every state in $K_0(A)$ is induced by a faithful trace (this happens for instance if $A$ is simple), then we can choose $\tau$ to be faithful.
 \end{remark}

The next proposition allows us to pass to unitizations:
\begin{proposition}\label{unitization}
If $\widetilde{A}$ has the $K_{0}$-embedding property, then $A$ has it, too.
\end{proposition}
\begin{proof}
Observe that $\iota :A\hookrightarrow \widetilde{A}$  ($\iota $ is the natural inclusion) has the property that $\iota_*(x)\geq 0$ iff $x\geq 0$ (see \cite[Chapter 4]{Rrdam2000AnIT} for instance), hence sends singular elements to singular elements. Then our assumption yields the result immediately.
\end{proof}

We close the section with some basic examples of $C^*$ algebras with the $K_0$-embedding property. \par
\begin{enumerate}[label=\roman*.]
    \item \emph{AF algebras} \cite[Lemma 1.14]{SPIELBERG1988325}.
    \item $A\rtimes \Z$, where $A$ is an AF algebra, provided that $A\rtimes \Z$ is quasidiagonal.\cite[Thm 5.5]{brown1998af}.

    \item Every  $C^*$-algebra $A$ for which there exists $D$ quasidiagonal with $K_0(D)$ being a torsion group and $\rho: A\hookrightarrow D$ faithful *-homomorphism. This is a direct corollary of Proposition \ref{UHF-stability}. Note that $K_0(A\otimes \U)=0$.\par
    This yields more examples like \emph{suspensions}, \emph{cones} and more generally
    \emph{exact and connective $C^*$ algebras}. This happens because if $A$ is exact and connective then by \cite[Thm A]{gabeembeddings}, $A$ embeds to the R\o rdam algebra $\mathcal{A}_{[0,1]}$ (see \cite{MR2063025} for the construction and properties of $\mathcal{A}_{[0,1]}$). Note that $ \mathcal{A}_{[0,1]}$ is quasidiagonal and has trivial K-theory.
    \item \emph{Separable, quasidiagonal, $C^*$ algebras with totally ordered $K_0$ group} (This is obvious from the definition). 
\end{enumerate}

\section{\texorpdfstring{$K_0$}{embedding Property is a local property} embedding Property is a local property}
Our main goal of the section is to show that $K_0$-embedding property is a local property. This is a very crucial result, because it will help us show that the property is satisfied by a wide variety of $C^*$-algebras, including the approximate subhomogeneous ones. Our first step is to show that in order to "kill" singular subgroups via an embedding, it is enough to do it via a sequence of asymptotically isometric and asymptotically multiplicative ccp maps. This is important, because, in order to pass from locally to globally, we need to extend maps. Although we can't in general extend *-homomorphisms, we can use Arveson's extension Theorem \cite[Thm 1.6.1]{brown2008textrm} to extend ccp maps defined on some subalgebra, to the whole algebra. But first we need the following proposition, which is almost identical to \cite[Prop. 2.5]{brown2004extensions}. \par
Recall that a $C^*$-algebra is called \textbf{MF} if there exists a faithful *-homomorphism $\phi:A\hookrightarrow \prod_{i=1}^{\infty} M_{k(n)}\slash \bigoplus_{i=1}^{\infty}M_{k(n)}$. Notice that a quasidiagonal $C^*$-algebra is MF. The converse is not true in general (for instance $C^*_r(\mathbb{F}_2)$ is MF but not quasidiagonal). However the Choi-Effros lifting Theorem \cite[Thm. 3.10]{choi-effros} implies that a nuclear, MF algebra is quasidiagonal. \par
For an extension $\eta$, $E(\eta)$ will be the $C^*$-algebra in the middle. Moreover, with $\mathcal{M}(A)$ we will denote the multiplier algebra of $A$ and we will use the notation $Q(A)=\mathcal{M(A)}\slash A$.

\begin{proposition}\label{MF-extensions}
Let $$ (\eta): \hspace{6mm} \begin{tikzcd}
0 \arrow[r] & A\otimes \K \arrow [r] & E \arrow[r] & B \arrow[r] & 0
\end{tikzcd}$$ be an extension, where $B$ is separable, nuclear and quasidiagonal, $A$ is MF and $\sigma$-unital (i.e it has a countable approximate unit) and $[\eta]=0 \in Ext(B, A\otimes \K)$. Then $E$ is MF.
\begin{proof}
Let $\rho:B\rightarrow B(H)$ be a faithful representation such that $H$ is separable, $\rho(B)\cap \K(H)=\{0\}$ and the orthogonal complement of the non-degeneracy subspace of $\rho(B)$ is infinite dimensional. Let $\tau$ be the extension from \cite[Thm 2.3]{brown2004extensions}. Then $\tau$ is both trivial and absorbing. By \cite[Lemma 2.2]{brown2004extensions}, $$E\hookrightarrow E(\eta \oplus \tau).$$ Because $[\eta]=0\in Ext(B,A\otimes \K)$ and $\tau$ is absorbing,
$$E(\eta \oplus \tau)\cong E(\tau).$$ Moreover, by looking at the statement of \cite[Thm 2.3]{brown2004extensions}, we can see that there exists an embedding 
$$E(\tau)\hookrightarrow (\rho(B)+\K(H))\otimes \widetilde{A}.$$ Thus $E\hookrightarrow (\rho(B)+\K(H))\otimes \widetilde{A}$. But $\rho(B)+\K(H)$ is quasidiagonal from \cite[Thm. 7.2.5]{brown2008textrm} and also nuclear, hence exact. So $(\rho(B)+\K(H))\otimes \widetilde{A}$ is MF by \cite[Prop 3.6]{MR3915315}. Thus $E$ is MF.
\end{proof}
\end{proposition}
Now we are ready to show the result we promised.
\begin{proposition}\label{MF-reducing}
 Assume that $A$ is separable, nuclear, quasidiagonal, and for some $G\leq K_0(A)$ singular, there exists a faithful *-homomorphism $\rho:A \hookrightarrow \prod M_{k(n)} \slash \bigoplus M_{k(n)}$ such that $\rho_*(G)=0$. Then there exists a quasidiagonal $C^*$-algebra $B$ and a faithful *-homomorphism $\phi:A\rightarrow B$, such that $\phi_*(G)=0.$
 \begin{proof}
 The proof of this proposition is essentially contained in the proof of \cite[Thm 4.11]{brown2004extensions}. However, for the sake of completion we will repeat the arguments here. Let $G$ be a singular subgroup of $K_0(A)$. Because $\oplus_{\N}C(\TT)$ satisfies the UCT, there exists a short exact sequence
 $$\begin{tikzcd}
 0\arrow{r} & A\otimes \K \arrow{r} & E \arrow{r}  & \oplus_{\N} C(\TT) \arrow{r} &0
 \end{tikzcd}$$
 satisfying $\delta_1(K_1(\oplus_{\N}C(\TT)))=G$.  Recall that $\delta_1$ is the boundary map in the K-theory six term exact sequence. 
 By assumption there exists an embedding $\rho_0:A \hookrightarrow \prod M_{k(n)} \slash \bigoplus M_{k(n)}$ such that $(\rho_0)_*(G)=0$. Let $D\subset \prod M_{k(n)} \slash \bigoplus M_{k(n)}$ be the hereditary $C^*$-subalgebra generated by $\rho_0(A)$. Then the *-homomorphism $\rho:A\otimes \K \rightarrow D\otimes \K$ satisfying $\rho(a\otimes b)=\rho_0(a)\otimes b$, is approximately unital (see \cite[Def 3.1]{brown2004extensions} for the definition of this property and the paragraphs below for an explanation of why this is true). Thus $\mathcal{M}(A\otimes \K)\subset \mathcal{M}(D\otimes \K)$ by [Ped. 3.12.12]. Hence there exists a Busby invariant $\eta_2: \oplus_{\N} C(\TT)\rightarrow Q(D\otimes \K)$ and an embedding $E\hookrightarrow E(\eta_2)$ such that the following diagram
  $$\begin{tikzcd}
 (\eta_1): \hspace{6mm} 0\arrow{r} & A\otimes \K \arrow[r, "\iota"] \arrow[hookrightarrow,d] & E \arrow{r} \arrow [hookrightarrow,d] & \oplus_{\N} C(\TT) \arrow{r} \arrow[d, "id"] &0 \\
(\eta_2):\hspace{6mm} 0\arrow{r} & D\otimes \K \arrow{r} & E(\eta_2) \arrow{r} & \oplus_{\N} C(\TT) \arrow{r} & 0
 \end{tikzcd}$$
 is commutative. Moreover, $K_1(D)=0$ by the proof of \cite[Lemma 3.2]{brown2004extensions} and $\rho_*(G)=0$ by \cite[Lemma 4.5]{brown2004extensions} and the fact that $\prod M_{k(n)} \slash \bigoplus M_{k(n)}$ has real rank zero and cancellation of projections. Hence both boundary maps on the bottom short exact sequence are zero. Indeed $K_1(D)=0$ implies $\delta_0=0$, while $\delta_1=0$ because of $\rho_*(G)=0$ and the naturality of the six term exact sequence. From the UCT and the fact that $K_i(\oplus_{\N} C(\TT))$ is a free $\Z$-module for $i=0,1$, we get that $[\eta_2]=0\in Ext(\oplus_{\N} C(\TT), D\otimes \K)$. Moreover, the fact that $A$ is separable, implies that $D \otimes \K$ is $\sigma$-unital.  So, by Proposition \ref{MF-extensions} $E(\eta_2)$ is an MF-algebra. Because MF-algebras are closed under taking subalgebras and every nuclear MF algebra is quasidiagonal, $E$ is quasidiagonal. Finally by the 6-term exact sequence, $\iota_*(G)=0$, so $A\hookrightarrow A\otimes \K \hookrightarrow E$ is the desired embedding.
 \end{proof}
 \end{proposition}

 The next step is to reduce to "killing" finitely generated singular subgroups.
 \begin{proposition}\label{f.g-reducing}
 Let $A$ be separable, unital, nuclear, quasidiagonal and assume that for every finitely generated singular subgroup $G$ of $K_0(A)$, there exists a faithful *-homomorphism $\rho: A \rightarrow B$, where B is quasidiagonal and $\rho_*(G)=0$. Then $A$ has the $K_0$-embedding property.
 \begin{proof}
 Let $G\leq K_0(A)$ be any singular subgroup. Because $G$ is countable, there is an increasing sequence of singular and finitely generated subgroups $G_n \leq K_0(A)$, such that $\bigcup_{n=1}^{\infty} G_n=G$. By assumption, there are faithful *-homomorphisms $\rho_n:A\rightarrow B_n$, where $B_n$ is quasidiagonal and $(\rho_n)_*(G_n)=0$.
 Because $B_n$ is quasidiagonal, for each $n$ there exists a sequence of ccp, asymptotically isometric and asymptotically multiplicative maps $\phi_{mn}: B_n\rightarrow M_{k_{mn}}$. Set $\psi_n=\phi_{d(n)n}$. Observe that we can take the $d(n)$'s to be large enough, so that $(\psi_n \circ \rho_n)_{n\in \N}$ is asymptotically multiplicative and asymptotically isometric. Thus if $$\rho:A\rightarrow \prod M_{l(n)} \slash \bigoplus M_{l(n)} \text{ with } \rho(x)=(\psi_n \circ \rho_n(x))_{n\in \N} (\text{where } l(n)=k_{d(n)n})$$ then  $\rho$ is a faithful *-homomorphism.
  Let now $g \in G$. Then there is $n_0 \in \N \text{ such that } \hspace{2mm} g\in G_n \hspace{2mm} \forall n\geq n_0$.
  Hence $(\rho_n)_*(g)=0$, which implies $ (\psi_n \circ \rho_n)_*(g)=0\hspace{2mm} \forall n\geq n_0$, so $\rho_*(g)=0$ in $K_0(\prod M_{l(n)}\slash \bigoplus M_{l(n)})= l^{\infty}(\Z)\slash c_{00}(\Z)$. Because $g$ is arbitrary, it follows that $\rho_*(G)=0$. Finally, Proposition \ref{MF-reducing} applies to yield that $A$ has the $K_0$-embedding property.
 \end{proof}
 \end{proposition}

 Now we have the tools to show that $K_0$-embedding property is a local property. But first we need to say explicitly what we mean by this.

 \begin{definition}[Def. 1.5, \cite{elliott2020decomposition}]Let $\CC$ be a class of $C^*$-algebras and $A$ be a $C^*$-algebra. We say that $A$ is \textbf{locally approximated by algebras in $\CC$} if for every finite subset $F\subset \subset A$ and for every $\varepsilon >0$, there exists a $C^*$-subalgebra $C\subset A$ such that $C\in \CC$ and $F\subset_{\varepsilon}C$.
 \end{definition}
 Assume that $A$ is locally approximated by algebras in $\CC$. Notice that by shrinking the class $\CC$, so that it contains only the $C^*$-algebras needed to guarantee local approximation, we may assume that $\CC$ consists only of $C^*$-subalgebras of $A$. Observe that  $\widetilde{A}$ is locally approximated by algebras in $\widetilde{\CC}$, where $\widetilde{\CC}=\{B+\C1_{\widetilde{A}} \hspace{3mm} | \hspace{3mm} B\in \CC\}$. This observation is important for managing the non-unital technicalities on the following proposition. Moreover, $A\otimes \U$ can be locally approximated by $\{\B\otimes \U \hspace{3mm} | \hspace{3mm} B\in \CC\}$. Note also that by the standard picture for $K_0$, every element of $K_0(A)$ is of the form $[p]_0-[s(p)]_0$, where $p\in P_m(\widetilde{A})$ and $s:\widetilde{A}\rightarrow \widetilde{A}$ satisfies $s(a+b\cdot 1)=b\cdot 1$ \cite[Prop. 4.2.2]{Rrdam2000AnIT}. \par
Let $h:A\rightarrow B$ be a ccp map that is $(F, \varepsilon)$ multiplicative, for some subset $F\subset P_{\infty}(A)$ and $0<\varepsilon<\frac{1}{4}$. Let $p\in F$ be a projection. Then, $||h(p)-h(p)^{2}||< \varepsilon$, hence \cite[Prop. 6.3.1]{Rrdam2000AnIT} implies that there exist $q\in P_{\infty}(B)$ such that $||h(p)-q||<2\varepsilon<\frac{1}{2}$. So, we can define $h_*([p]_0)=[q]_0$. Because every two projections that have distance $<1$ give rise to the same element in $K_0$, $h_*([p]_0)$ is well-defined. Moreover, observe that if $p_1,p_2,p_1\oplus p_2 \in F$, then $h_*([p_1]_0)+h_*([p_2]_0)=h_*([p_1\oplus p_2]_0)$. In this way, we can extend the notion of the induced $K_0$ map to sequences of asymptotically multiplicative ccp maps (for more details see, for instance, \cite[Def. 2.4 and below]{dadarlat2004morphisms}).

 \begin{proposition}\label{local_approximations}
 Let $A$ be a separable $C^*$-algebra that is locally approximated by algebras in $\CC$, where $\CC$ is a class that contains only separable, nuclear, quasidiagonal $C^*$-algebras with the $K_0$-embedding property. Then $A$ has the $K_0$-embedding property.
 \begin{proof}
 Notice that because nuclearity and quasidiagonality are both local properties, $A$ satisfies them. (For nuclearity, it is \cite[Ex. 2.3.7]{brown2008textrm}. For quasidiagonality, although we couldn't find an explicit reference, the result is well-known. Also the reader can deduce that the proof of this proposition implies that quasidiagonality is a local property). \par 
 Because of Proposition \ref{UHF-stability} and the fact that separability, nuclearity and quasidiagonality are preserved under tensoring with $\U$, we may assume that $A$ is $\U$-stable. Let $G\leq K_0(A)$ be finitely generated and singular.
 Let $$G=span_{\Z}\{g_1,g_2,...,g_m\}, \text{ where } g_i=[p^{(0)}_i]_0-[s(p^{(0)}_i)]_0, \text{ with } p_i^{(0)} \in P_w(\widetilde{A}).$$ Note that because $K_0(A)$ is torsion free and the structure theorem of finitely generated abelian groups, $G$ is a free $\Z$-module, hence the $g_i$'s are linearly independent. Because, nuclearity, quasidiagonality, separability and the ordered $K_0$ group are not affected when passing to matrix algebras, we may assume that $p^{(0)}_i \in \widetilde{A}$.  Set $$P=\{p^{(0)}_i,i=1,2,...,m\}$$ and fix an increasing sequence $P\subset F_n \subset \subset \widetilde{A}$ with $\overline{\bigcup F_n}=\widetilde{A}$, $\varepsilon_n \rightarrow 0$, $\varepsilon_n<\frac{1}{80}$ for every $n$. \par
 By assumption (and the comments before the statement of this proposition) there exist $C_n\in \CC, C_n\subset A$ such that $F_n\subset_{\varepsilon_n}\widetilde{C_n}$, where $\widetilde{C_n}=C_n+\C1_{\widetilde{A}}$. By \cite[Lemma 6.3.1]{Rrdam2000AnIT} we can find $p_{in} \in P(\widetilde{C_n})$ such that $||p_{in}-p_i^{(0)}||<1$. Moreover, we can find a sequence $(\overline{F_n})_n$ of finite subsets of $\widetilde{C_n}$ such that $p_{in}\in \overline{F_n}$ and for every $a\in F_n$ there exists $b\in \overline{F_n}: \vert\vert b-a \vert \vert <\varepsilon_n<\frac{1}{80}$.\par
 Observe that $g_i=[p_{in}]_0-[s(p_{in})]_0$ in $K_0(A)$. Let $\iota:C_n\hookrightarrow A$ be the natural embedding. Consider $L:=span_{\Z}\{l_1,l_2,...,l_m\}$, for some $l_i$ that satisfy $\iota_*(l_i)=g_i$. Such $l_i$ indeed exist because  $p_{in} \in P(\widetilde{C_n})$. Note that these need not be unique, because $\iota_*$ is not in general injective. Assume that there exists $v\in L$ such that $v>0$. Then $\iota_*(v)\geq 0$. But we have a contradiction, because $\iota_*(v)\in G$, which is singular, and $\iota_*(v)\neq 0$ by construction of $L$ and the fact that the $g_i$'s are linearly independent. Thus, we can view $G$ as a singular subgroup of $K_0(C_n)$, or even as a singular subgroup of $K_0(\widetilde{C_n})$ (see Proposition \ref{unitization}). \par
 By assumption, $C_n$ has the $K_0$-embedding Property, so there exists $\phi_n:\widetilde{C_n}\rightarrow M_{k(n)}$ ucp such that:
\begin{equation}
    \label{4.1}
    (\phi_n) \text{ is }(\overline{F_n},\varepsilon_n) \text{ multiplicative, }\vert \vert \phi_n(a)\vert \vert \geq \vert \vert a \vert \vert -\frac{10\varepsilon_n}{12} \text{ for every } a\in \overline{F_n} \text{ and } (\phi_n)_*(G)=0.
\end{equation}

 By Arveson's Extension Theorem, we can extend $\phi_n$ to $\overline{\phi}_n:\widetilde{A}\rightarrow M_{k(n)}$ ucp. Of course, $(\overline{\phi}_n)_*(G)=0$. \par
 Let $\varepsilon>0, x,y\in A$. Then there exists $n_0=n_0(\varepsilon)$ with the following two properties:
 \begin{equation}
     \label{3.2}
     \forall N\geq n_0, \varepsilon_N<\frac{\varepsilon}{24} \hspace{10mm} \text{ and }
\end{equation}
\begin{equation}
\label{3.3}
\exists x^{(0)}_n,y^{(0)}_n \in F_{n_0}\subset F_N \text{ such that }\vert \vert x^{(0)}_n-x \vert \vert <\frac{\varepsilon}{24}, \vert \vert y^{(0)}_n-y \vert \vert <\frac{\varepsilon}{24}.
\end{equation}

Moreover, there exist $x_N,y_N\in \overline{F_N}$ such that
$$||x_N-x_n^{(0)}||<\varepsilon_N \hspace{4mm}\text{and} \hspace{4mm} ||y_N-y_n^{(0)}||<\varepsilon_N.$$ Thus
\begin{equation}
    \label{3.4}
    \vert \vert x_N-x \vert \vert <\frac{\varepsilon}{12}, \vert \vert y_N-y \vert \vert <\frac{\varepsilon}{12}.
\end{equation}
Now if $N\geq n_0$, it follows that
$$ ||\overline{\phi}_N(xy)-\overline{\phi}_N(x)\overline{\phi}_N(y)||\leq$$
$$\leq||\overline{\phi}_N(x_Ny_N)-\overline{\phi}_N(x_N)\overline{\phi}_N(y_N)||+||\overline{\phi}_N(xy)-\overline{\phi}_N(x_Ny_N)||+||\overline{\phi}_N(x)\overline{\phi}_N(y)-\overline{\phi}_N(x_N)\overline{\phi}_N(y_N)||<$$
$$<\varepsilon_N+||xy-x_Ny_N||+||\overline{\phi}_N(x)\overline{\phi}_N(y)-\overline{\phi}_N(x_N)\overline{\phi}_N(y_N)||<$$
$$<\varepsilon_N+4\frac{\varepsilon}{12}+4\frac{\varepsilon}{12}<\varepsilon$$
where on the second inequality we use (3.1), on the third we use (3.4) and on the fourth we use (3.2). The fact that $\overline{\phi}$ is contractive is used on the second and third inequalities. Moreover, we have
$$||\overline{\phi}_N(x)||\geq ||\overline{\phi}_N(x_N)||-||\overline{\phi}_N(x-x_N)||\geq$$ $$\geq||x_N||-\frac{10\varepsilon}{12}-||x-x_N||\geq ||x||-\frac{\varepsilon}{12}-\frac{10\varepsilon}{12}-\frac{\varepsilon}{12}=||x||-\varepsilon.$$

Hence $(\overline{\phi}_n)_{n\in \N}$ is asymptotically multiplicative and isometrically isometric. The result yields from Proposition \ref{MF-reducing} and Proposition \ref{f.g-reducing}.
 \end{proof}
 \end{proposition}
 \begin{remark}\label{inductive_limit}
 Let $\CC$ be a class that contains only nuclear, separable, quasidiagonal $C^*$-algebras that have the $K_0$-embedding property. Let also $A=\varinjlim A_n$, where $A_n\in \CC$ for every $n$. Assume that the connecting maps are injective. Then $A$ is locally approximated by algebras in $\CC$. Thus, by Proposition \ref{local_approximations}, $A$ has the $K_0$-embedding property.
 \end{remark}

 \section{Constructing AF embeddings with control on K-theory}
 Let $A$ be a unital, separable, nuclear, $C^*$-algebra with a faithful trace that satisfies the UCT and $G\leq K_0(A)$ a singular subgroup. By \cite[Thm.A]{Schafhauser2018SubalgebrasOS} $A$ is AF-embeddable. In order to show the $K_0$-embedding property for $A$, it is enough to find (for every such $G$) an AF-algebra $B$ and an embedding $\rho:A\hookrightarrow B$ such that $\rho_*(G)$ is singular. Indeed, because AF-algebras have the $K_0$ embedding property, there exists an embedding $\phi:B\hookrightarrow D$, where $D$ is a quasidiagonal $C^*$-algebra and $\phi_*(\rho_*(G))=0$. Our desired embedding is $\phi \circ \rho$. However, constructing $B$ and $\rho$ can be very difficult. On the other hand, it is easier to construct the $K_0$ map. Then, if $A$ is "nice enough" and the $K_0$ map is chosen suitably, we can "lift" it to the $C^*$-algebra level. Our main tool is the following Theorem, which is a direct Corollary of \cite[Cor 5.4]{Schafhauser2018SubalgebrasOS}. We would like to thank Jose Carrion and Chris Schafhauser for pointing us out how can it be deduced.

 \begin{theorem}(cf. [Cor. 5.4, \cite{Schafhauser2018SubalgebrasOS}]) \label{lifting_theorem} Assume that $A,B$ are unital, separable $C^*$-algebras such that: $A$ is nuclear and satisfies the UCT, $B$ is a $\U$-stable AF algebra that has a unique trace $\tau_B$. Assume also that there exist $\sigma : K_0(A)\rightarrow K_0(B)$ positive group homomorphism and $\tau_A \in T(A)$ faithful, such that $\sigma([1_A]_0)=[1_B]_0$ and $\hat{\tau}_A=\hat{\tau}_B\circ \sigma$.
Then there exists a unital, faithful *-homomorphism $\phi:A\rightarrow B$ such that $\phi_*=\sigma$.
\begin{proof}
Let $A,B,\tau_A, \tau_B$ and $\sigma$ as in the hypothesis. Set $$P:=\{x \in K_0(B) \hspace{2mm} \vert \hspace{2mm} \hat{\tau}_B(x)>0 \}\cup \{0\}.$$ Then $(K_0(B),P,[1_B]_0)$ is a simple ordered dimension group with unique state, so by the Effros-Handelman-Shen Theorem \cite[Thm. 7.2.6]{Rrdam2000AnIT} as well as \cite[Cor. 1.5.4, Prop. 1.5.5]{rordam2002classification}, there exists a unital, separable, simple AF algebra $C$ with unique trace $\tau_C$ such that $(K_0(C),K_0(C)^+,[1_C]_0) \cong (K_0(B),P,[1_B]_0).$ via an order isomorphism $\gamma: K_0(C)\rightarrow K_0(B)$. Denote with $\beta:(K_0(C), K_0(C)^+)\rightarrow(K_0(B),K_0(B)^+)$ the map that satisfies $\beta(x)=\gamma(x)$ for every $x\in K_0(C)$. Denote also with $\alpha:K_0(A)\rightarrow K_0(C)$ the (unique) group homomorphism such that $\beta\circ \alpha= \sigma$. Let $y\in K_0(A)^+ \backslash\{0\}.$ Because $\tau_A$ is faithful, $\hat{\tau}_A(y)>0$, which implies $ \hat{\tau}_B(\sigma(y))>0$, so $ \sigma(y)\in P\backslash \{0\}$, which implies that $\gamma^{-1}\circ \sigma(y)\in K_0(C)^+\backslash \{0\}$. Because $\gamma^{-1}$ and $\beta$ are set theoretic inverses, we get $a(y)=\gamma^{-1} \circ \sigma(y)$, so $\alpha$ is a positive group homomorphism. Moreover $\alpha[1_A]_0=[1_C]_0$. By \cite[Cor. 5.4]{Schafhauser2018SubalgebrasOS}, there is a unital, faithful *-homomorphism $\rho: A\rightarrow C$ such that $\rho_*=\alpha$. Let $x\in K_0(C)\backslash \{0\}$. Then $\hat{\tau}_C(x)>0$ which implies $\hat{\tau}_B(\beta(x))>0$. Because $\hat{\tau}_B$ is the unique state in $(K_0(B), K_0(B)^+)$, it follows from \cite[Cor. 4.13]{Goodearl1986PartiallyOA} that $\beta(x)>0$. So $\beta$ is a positive group homomorphism. Moreover $\beta[1_C]_0=[1_B]_0$. By Elliott's classification of AF algebras \cite[Thm. 7.3.4]{Rrdam2000AnIT}, there exists a unital *-homomorphism $\psi:C\rightarrow B$ such that $\psi_*=\beta$. Because injective and $C$ simple, $\psi$ is automatically faithful. Thus $\phi:=\psi \circ \rho: A\rightarrow B$ is a unital, faithful *-homomorphism such that $\phi_*=\sigma.$
\end{proof}
\end{theorem}

A natural question to ask is whether $K_0$-embedding property is preserved under direct sums. \par
Let $A,B$ separable, unital, nuclear and quasidiagonal with the $K_0$-embedding property and $G\leq K_0(A\oplus B)\simeq K_0(A)\oplus K_0(B)$ be a singular subgroup. If $G=G_1\oplus G_2$, where $G_1\leq K_0(A), G_2\leq K_0(B)$ singular subgroups, then everything works fine. The problem, however, is that $G$ can be way more complicated. Let $x\in K_0(A)$ singular with the property that $mx\neq 0$ for every $m\in \N^*$. Then  $(x,-[1]_0),(x,[1]_0), (x,0)$ are all singular. Thus, if $A\oplus B$ has the $K_0$-embedding property, then there have to be $\phi_i:A\hookrightarrow B_i,i=1,2,3$, where $B_i$ is quasidiagonal, such that $(\phi_1)_*(x)>0$, $(\phi_2)_*(x)<0$ and $(\phi_3)_*(x)=0$. We will show that this indeed happens if there exists $\tau \in T(A)$ faithful such that $\hat{\tau}(x)=0$ (Proposition \ref{control_K_0}). But first we need a simple lemma.

\begin{lemma}\label{total_order}
Let $G$ be a countable abelian group that is also a $\Q$-vector space. Then there exists a total order on $G$.
\begin{proof}
 Because $G$ is countable and a $\Q$-vector space, it has a countable (Hamel)
 basis, call it $\mathcal{B}$. If the basis is finite, then $G\cong \Q^n$ for some $n\in \N$, while if the basis is countably infinite, then $G\cong c_{00}(\Q)$. In any case, we may see the elements of $G$ as sequences. We will put an order as follows: \par
 If $x=(x_n)_n$ and $y=(y_n)_n$, $x,y\in G$, then we initially look at the first coordinate. If
 $x_1<y_1$, set $x\preceq y$. If $x_1>y_1$, set $y\preceq x$. If $x_1=y_1$, we look at the
 second coordinate. If $x_2<y_2$, set $x\preceq y$. If $x_2>y_2$, set $y\preceq x$. If $x_2=y_2$,
 we look at the third coordinate, etc. In this way we have defined a total order $\preceq$ on $G$. Note that this order is heavily used and is called the lexicographic order.
\end{proof}
\end{lemma}

Now we can prove the result we promised. \par

\begin{proposition}\label{control_K_0}
Let $A$ be a separable, unital, nuclear $C^*$-algebra that satisfies the UCT. Let also $\tau \in T(A)$ be a faithful trace and $x\in K_0(A)$ such that $x$ is non-torsion and $\hat{\tau}(x)=0$. Then, there exist faithful, unit preserving *-homomorphisms $\phi_i:A\rightarrow B_i$, where $i=1,2,3, \text{ and }  B_i$ are unital AF algebras, such that $$(\phi_1)_*(x)>0, \hspace{4mm} (\phi_2)_*(x)<0 \hspace{4mm} \text{ and } \hspace{2mm} (\phi_3)_*(x)=0$$
\begin{proof}
After tensoring with the universal UHF algebra, Proposition \ref{UHF-stability} allows us to assume that $A$ is $\U$-stable. Then $K_0(A)$ is countable and also a $\Q$-vector space. Hence, by Lemma \ref{total_order}, there exists a total order on $K_0(A)$, call it $\succeq$. Fix a faithful trace $\tau \in T(A)$ and set
$$P=\{a\in K_0(A) : \hat{\tau}(a)>0\}\cup \{a\in K_0(A): \hat{\tau}(a)=0 \text{ and }  a\succeq0\}\subset K_0(A).$$ Observe that because $\succeq$ is a total order, $(K_0(A),P, [1_A]_0)$ is a scaled, totally ordered (hence dimension) group. Indeed, let $a,b \in P$. Then $\hat{\tau}(a),\hat{\tau}(b)\geq 0$. Thus $\hat{\tau}(a+b)\geq 0$. If $\hat{\tau}(a+b)>0$, then by construction $a+b\in P$. If $\hat{\tau}(a+b)=0$, we must have $\hat{\tau}(a)=\hat{\tau}(b)=0$ which implies $ a\succeq 0 \text { and } b\succeq 0$. Thus $a+b\succeq 0$, which, along with $\hat{\tau}(a+b)=0$, implies that $ a+b \in P$. Moreover, if $a \in P \cap -P$, we deduce that $ \hat {\tau}(a)\geq 0$ and $ \hat{\tau}(a)\leq 0$. Hence $\hat{\tau}(a)=0 $. So $a \in P \text{ implies } a\succeq 0$ while $a\in -P \text{ implies } 0 \succeq a$. Thus $a=0$. This shows that $P \cap -P=\{0\}$. If  $a\in K_0(A)$, then we have three cases. We either have $\hat{\tau}(a)>0$, which yields $a\in P$, or $\hat{\tau}(a)<0$, which yields $a\in -P$, or $\hat{\tau}(a)=0$. In this case, if $a\succeq 0$, then $a\in P$, while if $0\succeq a$, it follows that $a \in -P$. So, the order is total. Moreover, if $x\in K_0(A)^+\backslash \{0\}$, then, because $\tau$ is faithful, we have that $\hat{\tau}(x)>0$ which implies $ x\in P$. So $K_0(A)^+ \subset P$. Hence, $[1_A]_0$ is still an order unit and $$\beta: (K_0(A),K_0(A)^+,[1_A]_0)\rightarrow (K_0(A),P, [1_A]_0)$$ is a positive group homomorphism sending the order unit to the order unit.\par
\textbf{Claim:} \hspace{2mm} $S(K_0(A),P, [1_A]_0)=\{\hat{\tau}\}$. \par
\textbf{Proof of Claim:} \hspace{2mm} Let $\rho \in S(K_0(A),P, [1_A]_0)$ and $y\in K_0(A)$. For better notation set also $u=[1_A]_0$. Assume that $\rho(y)\neq \hat {\tau}(y)$. Assume first that $\rho(y) < \hat{\tau}(y)$. Then $\exists q\in \Q: \rho(y)<q<\hat{\tau}(y).$ Thus $$\hat{\tau}(y-qu)>0 \Rightarrow y-qu\geq 0 \Rightarrow$$ $$ y\geq qu \xRightarrow{\rho\in S(K_0(A),P^+)}\rho(y)\geq \rho(qu) \Rightarrow \rho(y)\geq q,$$ contradiction. The case $\rho(y) > \hat{\tau}(y)$ is contradicted in an identical way. Because $y$ is arbitrary, $\rho=\hat{\tau}$. \qed \par
 By the Effros-Handelman-Shen Theorem \cite[Thm 7.2.6]{Rrdam2000AnIT}, there is a unital, separable, AF algebra $B_1$ such that \par
 $(K_0(B_1),K_0({B_1})^+,[1_{B_1}]_0)\cong(K_0(A),P, [1_A]_0) \text{  via an order isomorphism  } \gamma:K_0(A)\rightarrow K_0(B_1). $ Because of the Claim and the fact that the map $\Psi$ in (2.5) is a bijection, $B_1$ has a unique trace, call it $\tau_{B_1}$ that satisfies $\hat{\tau}_{B_1}\circ \beta=\hat{\tau}$. Also $B_1$ is $\U$-stable because $K_0(A)\cong K_0(A)\otimes \Q$. Indeed,  $$(K_0(B_1),K_0(B_1)^+,[1]_0)\cong (K_0(B_1\otimes \U),K_0(B_1 \otimes \U)^+,[1]_0)$$ which implies$$  B_1\cong B_1\otimes \U$$ by Elliott's Classification Theorem for AF algebras \cite[Thm. 7.3.4]{Rrdam2000AnIT}. Furthermore, we have the following commutative diagram:

 \begin{center}
 \begin{tikzcd}
 K_0(A) \arrow {r}{\gamma \circ \beta} \arrow[swap]{dr}{\hat{\tau}} & K_0(B_1) \arrow{d}{\hat{\tau}_{B_1}} \\
 & \R
 \end{tikzcd}
 \end{center}

 and we also have $\beta[1_{A}]_0=[1_{B_1}]_0$ \par
 So by Theorem \ref{lifting_theorem}, there is a unital, faithful *-homomorphism $\phi_1 :A \hookrightarrow B_1$ such that $(\phi_1)_*=\beta$. \par
 For $x,y\in K_0(A)$ we define a new order $\preceq$ such that $x\preceq y$ iff $y \succeq x$. Notice that $\preceq$ is also a total order on $K_0(A)$. If we use $\preceq$, instead of $\succeq$ as our total order and do the same work as before, we can find $B_2$ unital, AF and $\phi_2: A\rightarrow B_2$ faithful *-homomorphism. Fix any $x\neq 0$ such that $\hat{\tau}(x)=0$. Then, out of $(\phi_1)_*(x)$ and $(\phi_2)_*(x)$, one is positive and the other negative. WLOG $x\succeq 0$. Then $(\phi_1)_*(x)>0$ and $(\phi_2)_*(x)<0$. Finally, the existence of $\phi_3$ is already known from the proof of \cite[Thm A]{Schafhauser2018SubalgebrasOS}.
\end{proof}
\end{proposition}

A natural question to ask is how this total order on $K_0(A)$ looks like. To get some idea, we will exhibit the following (basic) example:
\begin{example}
Let $A=\U \oplus \U$ and consider the faithful trace $\tau$ such that $\tau(a,b)=\frac{\sigma(a)+\sigma(b)}{2}$, where $\sigma$ is the unique trace on $\U$. Let's also use the following total order on $K_0(A)=\Q \oplus\Q$ (which will help us order the elements in $\ker(\hat{\tau})$): Define $(x,y)\succeq (a,b)$ iff either $x>a$ or $x=a$ and $y\geq b$. Then $P=\{(a,b) \hspace{2mm} \vert \hspace{2mm} a+b>0\} \cup \{(a,-a) \hspace{2mm} \vert \hspace{2mm} a\geq 0\}$. So if we see $\Q \oplus \Q$ as the points on the Euclidean plane with rational coefficients, then $P$ contains everything to the right of the line with equation $x+y=0$ plus the bottom part of the line.
\end{example}
The following corollary is immediate from the proof of Proposition \ref{control_K_0}, but we write it down, as it has its own independent interest, and we will also need it in Section 6.

\begin{corollary}\label{AF-embedding}
Let A be a separable, unital, nuclear $C^*$-algebra that satisfies the UCT and $\tau \in T(A)$ is faithful trace. Then there exists a unital, faithful *-homomorphism $\phi: A\rightarrow B$, where B is unital AF algebra that is $\U$-stable, and $\phi_*(x)$ is non-zero singular for every $x\in K_0(A)$ non-torsion with $\hat{\tau}(x)=0$.
\begin{proof}
Take $\phi=\phi_1 \oplus \phi_2$, where $\phi_1, \phi_2$ are as in the proof of Proposition \ref{control_K_0}.
\end{proof}
\end{corollary}

From the result of Corollary 4.5, the following question arises naturally: If $A$ is a separable, quasidiagonal, nuclear, unital $C^*$-algebra and $x\in K_0(A)$ singular, when is it true that $x$ can be "killed" by a state induced by a faithful trace? If $A$ is also simple, this always happens because of Remark \ref{trace_kills_singular}. Unfortunately, in the non-simple case this fails even for a commutative $C^*$-algebra.

\begin{example}\label{counterexample}
Consider $A=C(S^2)\oplus \C$. \par
In \cite[Example 6.3.4, (d)]{blackadar1998k}, it is shown that
$$K_0(C(S^2))=\Z^2, K_0(C(S^2))^+=\{(x,y), x>0\} \cup \{(0,0)\} \text{ and } [1_{C(S^2)}]=(1,0)$$
Set $x=(0,1,-1)\in K_0(A)$. Obviously $\Z x \cap K_0(A)^+=\{0\}$. \par
Moreover, $u=(1,0,1)$ is an order unit of $K_0(A)$ and if $y=(0,1,0)$, it follows that $u+ny>0 \hspace{2mm} \forall n\in\Z$. By (2.6) and (2.7), we have that
\begin{equation}
    \label{3.1}
    \rho(y)=0 \hspace{2mm} \forall \rho\in S(K_0(A)).
\end{equation}
Assume  that $\rho_0(x)=0$ for some $\rho_0 \in S(K_0(A))$. Then (4.1) implies $$0=\rho_0(x)=\rho_0(x-y)=\rho_0(0,0,-1).$$
But $(0,0,-1)< 0$, so $\rho_0$ cannot be induced by a faithful trace. \par

\end{example}

  \section{New examples of \texorpdfstring{$C^*$}{1}-algebras with the \texorpdfstring{$K_0$}{1}-embedding Property}
 Let $\mathcal{G}$ be the class of all separable, unital, nuclear, quasidiagonal $C^*$-algebras that satisfy the UCT, and have the property that every state in $K_0(A)$ is induced by a faithful trace. \par
 Because $$T(A\otimes \U)=\{\tau \otimes \sigma \hspace{2mm} \vert \hspace{2mm} \tau \in T(A)\},$$
 where $\sigma$ is the unique trace on $\U$, it follows that $A\in \mathcal{G}$ iff $A\otimes \U \in \mathcal{G}$.

A natural question to ask is how big this class is. \par
First of all, it has to be noted that by \cite[Thm A]{Schafhauser2018SubalgebrasOS} all the algebras in the class, are embeddable (in a unit preserving way) into simple, AF algebras.
\begin{proposition}\label{easy_case}
Every $A\in \mathcal{G}$ has the $K_0$-embedding property.
\begin{proof}
Let $A\in \mathcal{G}$. After tensoring with the Universal UHF-algebra, we may assume that $A$ is $\U$-stable because of Proposition \ref{UHF-stability}. Recall that the aforementioned comments imply $A\otimes \U\in \mathcal{G}$. Fix a singular subgroup $H\leq K_0(A)$. By Lemma \ref{order_lemma} for $G=K_0(A)$, $H_1=H_2=H$ and the assumption that every state in $K_0(A)$  can be induced by a faithful trace, we deduce that there exists $\tau\in T(A)$ faithful, such that $\hat{\tau}(H)=0$ (see Remark \ref{trace_kills_singular}). Result follows from Corollary \ref{AF-embedding} and the fact that AF-algebras have the $K_0$-embedding Property.
\end{proof}
\end{proposition} 
We will now exhibit some examples of $C^*$-algebras in $\mathcal{G}.$
\begin{proposition}\label{connected}
If $X$ is a separable, compact, Hausdorff and connected topological space, then $C(X)\in \mathcal{G}$.
\begin{proof}
By \cite[Ex 6.10.3 and Cor 6.3.6]{blackadar1998k}, $K_0(C(X))$ is a simple ordered group with unique state, call it $\rho$. Because the map in (2.5) is onto, $\rho$ is induced by (any) faithful trace in $C(X)$. Finally it is well-known that $C(X)$ is nuclear, quasidiagonal and satisfies the UCT. Hence $C(X)\in \mathcal{G}$.
\end{proof}
\end{proposition}

\begin{proposition}\label{simple}
If A is separable, nuclear, unital, simple, satisfies the UCT and has a trace, then $A\in\mathcal{G}$.
\begin{proof}
First, notice that by \cite [Cor. B]{tikuisis2017quasidiagonality}, $A$ is automatically quasidiagonal. Because the map in (2.5) is onto, every state in $K_0(A)$ can be induced by a trace. But $A$ is simple, so all traces are automatically faithful. Result follows.
\end{proof}
\end{proposition}
\begin{proposition}\label{crossed_products}
Let $A$ be unital, separable, nuclear, satisfies the UCT and $\sigma:\Z\rightarrow Aut(A)$ be a minimal action. Assume that $A$ has a $\sigma$-invariant trace. Then $A\rtimes_{\sigma} \Z \in \mathcal{G}.$
\begin{proof}
$A\rtimes_{\sigma} \Z$ is unital and separable. By \cite[Thm 4.2.4]{brown2008textrm} $A\rtimes_{\sigma} \Z \simeq A\rtimes_{\sigma,r} \Z$ and $A\rtimes_{\sigma} \Z$ is nuclear. Because $A$ has a $\sigma$-invariant trace, $A\rtimes_{\sigma} \Z$ is quasidiagonal by \cite[Cor 6.5]{Schafhauser2018SubalgebrasOS}. Because the bootstrap class in closed under taking crossed products with $\Z$ \cite[22.3.5]{blackadar1998k}, $A\rtimes_{\sigma} \Z $ satisfies the UCT. Let now $\rho\in S(K_0(A\rtimes_{\sigma} \Z))$. Because $A\rtimes_{\sigma} \Z$ is nuclear hence exact, the map in (2.5) is onto, so there is $\tau \in T(A\rtimes_{\sigma} \Z)$ such that $\hat{\tau}=\rho.$ The restriction $\tau |_{A}$ is an invariant trace. Notice that $N_{\tau|_{A}}=\{x\in A \hspace{2mm} \vert \hspace{2mm} \tau|_A(x^*x)=0\}$ is a $\sigma$-invariant ideal. Because the action is minimal, $N_{\tau|_A}=0$, so $\tau|_A$ is faithful. Hence $\tau|_A \circ E \in T(A\rtimes_{\sigma} \Z)$, where $E:A\rtimes_{\sigma} \Z \rightarrow A$ is the conditional expectation that sends $\sum_{g\in \Z}a_g g$ to $a_0$, is also faithful. From the comments in \cite[p.84]{blackadar1998k}, we have that $\widehat{\tau|_A \circ E}=\rho$ (see appendix for a detailed proof of this). Thus $A\rtimes_{\sigma} \Z \in \mathcal{G}.$
\end{proof}
\end{proposition}

If the action $\sigma$ is trivial, then it is minimal iff $A$ is simple. In this case, $A\rtimes \Z=A\otimes C(\TT)$. Thus, if $A$ is simple, $A\otimes C(\TT)\in \mathcal{G}.$ Actually, in order to deduce that the tensor product is in $\mathcal{G}$, we can weaken the simplicity condition for $A$, and assume $A\in \mathcal{G}$ instead.
\begin{proposition}
Let $A\in \mathcal{G}$. Then $A\otimes C(\TT)\in \mathcal{G}.$
\begin{proof}
Assume that $A\in \mathcal{G}.$ There is a split exact sequence $$\begin{tikzcd}
0\arrow[r]& SA\arrow[r, "\iota"] & A\otimes C(\TT) \arrow[r, "\pi"] & A \arrow[r] &0
\end{tikzcd},$$
where we identify $A\otimes C(\TT)$ with $C(\TT,A)$, $\iota$ is the inclusion map and $\pi(f)=f(1)$.
So, $K_0(A\otimes C(\TT))\cong K_0(A)\oplus K_1(A)$. By \cite[Prop 5.7]{MR3915315}, we have that $([1_A]_0,b)\in K_0(A\otimes C(\TT))^+$ for every $b\in K_1(A)$. Hence $(0,b)\in \Inf K_0(A\otimes C(\TT))$ for every $b\in K_1(A)$. Let $\rho\in S(K_0(C(\TT)\otimes A)).$ Then $\rho(a,b)=\rho_0(a)$, where $\rho_0 \in S(K_0(A))$. By assumption, $\rho_0=\hat{\tau_0}$, where $\tau_0$ is a faithful trace on $A$. Observe that $\rho=\hat{\tau}$, where $\tau=\tau_0 \otimes \sigma$ and $\sigma$ is (any) faithful trace on $C(\TT)$. So $\tau$ is faithful and hence $A\otimes C(\TT)\in \mathcal{G}$.
\end{proof}
\end{proposition}
\vspace{5mm}

On the other hand, notice that non-trivial direct sums of unital, quasidiagonal $C^*$-algebras cannot be on $\mathcal{G}$. \vspace{5mm}

So, it is natural to set $$ \mathcal{O}=\{\oplus_{i=1}^{n}A_i \hspace{2mm} \vert \hspace{2mm} n\in \N^* \text{ and for every } i,\text{ either } A_i \in \mathcal{G} \text{ or }A_i \text{ is an AF algebra}\}.$$

Corollary \ref{AF-embedding} gives us an indication that the elements of this class should have the $K_0$ embedding property. We will show that this indeed happens.
\begin{proposition}\label{direct_sums}
If $A\in \mathcal{O}$, then $A$ has the $K_0$-embedding property.
\begin{proof}
Let $A\in \mathcal{O}$. Then $A=\oplus_{i=1}^{n} A_i$, for some $A_i\in \mathcal{G}$ or AF algebra. After tensoring with the Universal UHF-algebra,  Proposition \ref{UHF-stability}, allows us to assume that $A_i$ is $\U$-stable for every $A_i$. Recall that $K_0(A)=\oplus_{i=1}^n K_0(A_i)$. Let $G\leq K_0(A)$ singular. Because $A_i$ is separable, $K_0(A_i)$ is countable. Because of Proposition \ref{unitization} and its proof we may assume that if $A_i$ is an AF-algebra, then it is unital. Because every singular subgroup is contained in a maximally singular subgroup, we may assume that $G$ is maximally singular. Because of Lemma \ref{(2.2)-holds}, we have that $G$ satisfies (2.2). \par
We want to find an embedding $\phi$ of $A$ into an AF algebra $B$ such that $\phi_*(G)$ is singular. The easiest way to do this, is to find embeddings $\phi_i:A_i\hookrightarrow B_i$ and take $\phi$ to be their direct sum. Because it is not easy to construct the AF-algebras $B_i$ explicitly, we will first construct the maps on the $K_0$-level and then use Theorem \ref{lifting_theorem} to "lift" to the $C^*$-algebra level. Note that if some $A_i$ is an AF algebra, we might not be able to use Theorem \ref{lifting_theorem}, as it is even possible that $A_i$ does not have any faithful trace. However, in this case every positive group homomorphism $\pi_i: K_0(A_i)\rightarrow K_0(B_i)$ with $\pi_i[1]=[1]$, lifts to the $C^*$-algebra level (recall that $A_i$ is an AF-algebra in this case). So, we need to find positive group homomorphisms $\pi_i:K_0(A_i,K_0(A_i)^+,[1]_0)\rightarrow (H_i,P_i,u_i)$ with $\pi_i([1])=u_i$ and $(H_i,P_i)$ dimension groups. One way to construct a dimension group, is to construct a totally ordered group. We will also want to secure that $(H_i,P_i,u_i)$ has a unique state. Finally, we want to define traces $\tau_i\in T(A_i)$ such that the unique state on $(H_i,P_i,u_i)$ is the composition of $\hat{\tau}_i$ with $\pi_i$. \par
For every $i=1,2,...,n$ set $$G_i^{zero}=\{x_i \in K_0(A_i) \text{ such that }(0,0,...,x_i,0,...,0)\in G\}.$$
We have that $\ G_i^{zero}\cap K_0(A_i)^+=\{0\}$ and also $G_i^{zero}$ satisfies (2.2) because $G$ satisfies (2.2). Thus, if $H_i=K_0(A_i)\slash G_i^{zero}$ is endowed with the order as defined in (2.3), $(H_i,H_i^+,\overline{[1_{A_i}]}_0)$ is a scaled ordered, unperforated group by Lemma \ref{order_on_quotient} and $H_i\cong H_i\otimes \Q$. The latter holds because $K_0(A_i)\cong K_0(A_i)\otimes \Q$ and $G_i^{zero}\cong G_i^{zero}\otimes \Q$.\par
If we want to achieve our goal, we must kill the elements of $G_i^{zero}$.\par

Let $\pi_i:K_0(A_i)\rightarrow H_i$ be the quotient map and $$H=\oplus_{i=1}^n H_i, H^+=\oplus_{i=1}^n H_i^+, \pi=\oplus_{i=1}^n\pi_i:K_0(A)\rightarrow H.$$ \par
\vspace{5mm}
Our first claim allows us to make sure that after moving to the quotient our group is still maximally singular and there are no more nonzero elements of the form $(0,0,...,a,0,...,0)$. \par
\textbf{Claim 1:} $\pi(G)$ is maximally singular and if  $y_i\in H_i$ with $(0,...,0,y_i,0,...,0)\in \pi(G)$, then $y_i=0$.\par
\textbf{Proof of Claim 1:} Assume that there exists $y\in H^+\cap \pi(G)$ with $ y\neq 0$. Then
\begin{equation}
    \label{5.1}
    y=\pi(x) \text{ for some }x=(x_1,...,x_n)\in G.
\end{equation}
Because $y\in H^+$, for each $i=1,2,...,n$ there exists \begin{equation}
    \label{5.2}
    r_i\in G_i^{zero} : x_i+r_i \geq 0.
\end{equation}
Because $y\neq 0$, there exists $i_0$ such that
\begin{equation}
    \label{5.3}
    x_{i_0}+r_{i_0}>0.
\end{equation}
Because $ r_i\in G_i^{zero},$
\begin{equation}
    \label{5.4}
     (0,0...,0,r_i,0,...,0)\in G.
\end{equation}
(5.1),(5.2),(5.3),(5.4) yield that $$0<(x_1+r_1,...,x_n+r_n)\in G.$$
contradicting the fact that $G$ is singular. Hence, $\pi(G)$ is singular.\par
To show maximality assume for the sake of contradiction that \begin{equation}
    \label{5.5}
    L\supsetneq \pi(G), L\cap H^+=\{0\}.
\end{equation}
This yields that
\begin{equation}
\label{5.6}
\pi^{-1}(L)\cap K_0(A)^+=\{0\}.
\end{equation}
Indeed, if there exists $z\in \pi^{-1}(L),z>0$, then we have that $\pi(z)\in L$ and $\pi(z)\geq 0$, However, if $\pi(z)=0$, then $z=(z_1,...,z_n)$ for some $z_i\in G_i^{zero}$, which contradicts the fact that $z>0$. So, $\pi(z)>0$ which contradicts $L\cap H^+=\{0\}$.\par
(5.5),(5.6), along with the maximality of $G$ give that $\pi^{-1}(L)=G$ so $\pi(\pi^{-1}(L))=\pi(G)$ which implies $ L=\pi(G)$ contradiction. So, $\pi(G)$ is maximally singular.\par

To prove the last statement, assume for the sake of contradiction that \begin{equation}
    \label{5.7}
    (0,...,0,y_i,0,...0)\in \pi(G)
\end{equation}
where $y_i\in H_i$ for some $i$. Then there exist
\begin{equation}
    \label{5.8}
    x=(x_1,...,x_n)\in G
\end{equation}
such that \par
$$\pi(x)=(0,...,0,y_i,0,..,0).$$
If $j\neq i$, then $\pi_j(x_j)=0$. Thus $x_j\in G^{zero}_j$. Moreover, $\pi_i(x_i)=y_i$. \par

Because $x_j\in G_j^{zero}$, it follows that
\begin{equation}
    \label{5.9}
    (0,...,0,x_j,0,...,0)\in G.
\end{equation}
for every $j\neq i$. Finally, (5.8),(5.9)$\Rightarrow$ $(0,...,0,x_i,0,...,0)\in G\Rightarrow x_i\in G^{zero}_i\Rightarrow  y_i=0$ as desired.
\qed \vspace{5mm}

Set  $$H_i^{neg}=\{ x_i \in H_i, \text{ such that }\exists  ( a_{i},..., a_{i-1},  x_i,  a_{i+1},..., a_n)\in \pi(G), \text{ and } a_j \geq 0 \forall j\neq i\}.$$ and
$$H_i^{pos}=\{ x_i \in H_i, \text{ such that }\exists  ( a_{i},..., a_{i-1},  x_i,  a_{i+1},..., a_n)\in \pi(G), \text{ and } a_j \leq 0 \forall j\neq i\}.$$ \par
By construction of $H_i$,
\begin{equation}
    \label{5.10}
    H_i^{pos}\cap H_i^{neg}=\{0\}, \hspace{2mm} H_i^{pos}=-H_i^{neg},\hspace{2mm} H_i^{pos} \text{ and }H_i^{neg} \text{ are semigroups and } H_i^{neg}\cap H_i^+=\{0\}.
\end{equation}   \par
In order to achieve our goal, we must make (on the totally ordered group) the elements of $H_i^{pos}$ positive and the elements of $H_i^{neg}$ negative. \par
On the next claim, we will define our traces. \par
\textbf{Claim 2:}\hspace{2mm} For every $i=1,2,...,n$, there exists $\tau_i\in T(A_i)$, which can be taken faithful if $A_i\in \mathcal{G}$, such that
\begin{enumerate}[label=\roman*.]
    \item $\hat{\tau}_i(G_i^{zero})=0$.
    \item If $\bar \tau_i$ is the induced state on $H_i$, then $\bar \tau(z)\geq 0$ for every $z\in H_i^{pos}$ and thus $\bar \tau(z)\leq 0$ for every $z\in H_i^{neg}$.
\end{enumerate}

\textbf{Proof of Claim 2:}\hspace{2mm} The fact that $G_i^{zero}$ is singular,(5.10) and $K_0(A_i)\cong K_0(A_i)\otimes \Q$ guarantee that the assumptions of Lemma \ref{order_lemma} are satisfied for $G_i^{zero}=H_1$ and $\pi_i^{-1}(H_i^{pos})=H_2$. Result follows from the fact that the map in (2.5) is onto and Lemma \ref{order_lemma} itself. Note that if $A_i\in \mathcal{G}$ for some $i$, then every state can be induced by a faithful trace, which allows us to choose $\tau_i$ to be faithful (see Remark \ref{trace_kills_singular}).

\qed\par

\vspace{3mm}
Now, we are ready to define our total orders. \par
 Let $x_i\neq 0, x_i\in H_i$. Then by maximality of $\pi(G)$ and the fact that $x=(0,0,...,x_i,0,...0)\notin \pi(G)$, we deduce that $$\exists k\in \Z \text{ and }a=(a_1,a_2,...,a_n)\in \pi (G), \text{ such that }a+kx>0 \hspace{2mm}. \text{ Thus } $$\begin{equation}
\label{5}
a_i+kx_i\geq 0. \hspace{2mm}
\end{equation}
Observe that $\forall l\neq i,$ we have $ a_l\geq 0$, so \begin{equation}
    \label{6}
    a_i\in H_i^{neg}.
\end{equation}
Define $\Phi_i:H_i\rightarrow\{0,-1,1\} $ such that \par
\begin{align}
\Phi_i(x) = \left\{ \begin{array}{cc}
                0 & \hspace{5mm} \text{ if }x=0 \\
                1 & \hspace{5mm} if \hspace{2mm} x\neq 0 \text{ and } \exists k\in \N^*\hspace{2mm} and \hspace{2mm}a\in H_i^{neg}: a+kx\geq 0\\
                -1 & \hspace{5mm}  if \hspace{2mm} x\neq 0 \text{ and } \exists k\in \Z_{<0}\hspace{2mm} and \hspace{2mm}a\in H_i^{neg}: a+kx\geq 0\\
                \end{array} \right.
\end{align}
As expected, the elements that take value 1, will be our positive elements (on the totally ordered groups), while the elements that take value -1 will be our negative elements. In Claims 3-5 we will show the properties of $\Phi_i$ needed to make sure that we will end up with total orders. \vspace{2mm}

\textbf{Claim 3:}\hspace{2mm}$\Phi_i$ is well-defined. \par
\textbf{Proof of Claim 3:} First of all, (5.11) and (5.12) yield that $\forall x\in H_i, \Phi_i(x)$ can take at least one value, according to the definition. Assume that $\exists x\in H_i\backslash\{0\}, n_1>0,n_2<0 (n_i\in \Z), a,b\in H_i^{neg}$, such that $ a+n_1x\geq 0$ and $b+n_2x\geq 0$. Notice that if $a=0$ and $b=0$, we get that $x=0$, contradiction. So at least one of $a,b$ is nonzero. In addition, we have\vspace{2mm}

$\left\{
\begin{array}{rr}
     & a+n_1 x\geq 0\\
     &b+n_2 x\geq 0
\end{array}
\right.
\Rightarrow
\left\{
\begin{array}{rr}
     &  |n_2|a+n_1|n_2|x\geq 0 \\
     & n_1b+n_1n_2x\geq 0
\end{array}
\right.
\xRightarrow[\text{add together}]{n_2<0}
\vert n_2|a+n_1b\geq 0$. \vspace{3mm}

Because $a,b\in H_i^{neg}, n_1,n_2\neq 0 \text{ and by }(5.10)$, we deduce that $a=b=0$, contradiction.
So $\Phi_i$ is well-defined. \qed \vspace{2mm}

\textbf{Claim 4:} \hspace{2mm} \begin{enumerate}[label=\roman*.]
    \item If $\Phi_i(x)=\Phi_i(y)=1$, then $\Phi_i(x+y)=1$.
    \item If $\Phi_i(x)=\Phi_i(y)=-1$, then $\Phi_i(x+y)=-1$.
\end{enumerate}
\textbf{Proof of Claim 4:} \hspace{2mm} We will only prove the first statement. The second one can be proved in an identical way. \par
Assume that $\Phi_i(x)=\Phi_i(y)=1$. Then $x,y\neq 0$ and $\exists a,b\in H_i^{neg}, k,m\in \N^*$, such that \vspace{3mm}

$\left\{
\begin{array}{rr}
     & a+kx\geq 0\\
     & b+my\geq 0
\end{array}
\right.
\Rightarrow
\left\{
\begin{array}{rr}
     &  mkx+ma\geq 0\\
     & mky+kb\geq 0
\end{array}
\right.
\Rightarrow
 ma+kb+mk(x+y)\geq 0 (*).$

 If $x+y=0$(**), then it has to be $ma+kb\geq 0.$ Because $a,b\in H_i^{neg}, m,k\in \N^* \text{ and by }(5.10)$ we deduce that $a=b=0$. But this yields that $x,y\geq 0$. By (**) it follows that $x=y=0$, contradiction. So $x+y\neq 0$ and (*) gives $\Phi_i(x+y)=1$. \qed \vspace{2mm}

\textbf{Claim 5:}\hspace{2mm} $\Phi_i(-x)=-\Phi_i(x)$. \par
\textbf{Proof: }\hspace{2mm} This is obvious. \qed \par
\vspace{5mm}
Claim 6 will allow us to guarantee that the image on $G$ under the $K_0$ map of the AF-embedding will be singular.

\textbf{Claim 6:} There is no $ (x_1,x_2,...,x_n)\in \pi(G)\backslash \{0\}$ such that for every $i=1,2,...,n$, we have $\Phi_i(x_i)\geq 0$. (hence the same statement with negatives is true) \par
\textbf{Proof of Claim 6:} \hspace{2mm} Assume the contrary. Then there exists \begin{equation}
    \label{7}
    (x_1,x_2,...,x_n)\in \pi(G)\backslash\{0\} \text { such that } \forall i, \Phi_i(x_i)\geq 0.
    \end{equation}
Let $$\J:=\{i \hspace{2mm} \vert \hspace{2mm} x_i\neq 0\}.$$
 By assumption $\J\neq \emptyset$. By (\ref{7}), for each $i\in \J$, \begin{equation}
\label{**}
(a_1^{(i)},a_2^{(i)},...,,a_{i-1}^{(i)},m_ix_i+a_i^{(i)},a_{i+1}^{(i)},...,a_n^{(i)})> 0
\end{equation}
for some $m_i\in \N^*, a_i^{(i)}\in H_i^{neg}, a_j^{(i)}\geq 0 \hspace{2mm}\forall j\neq i$ and $(a_1^{(i)},... ,a_n^{(i)})\in \pi(G) $. It follows that \begin{equation}
    \label{8}
    l_i:=(\frac{N}{m_i}a_1^{(i)},...,\frac{N}{m_i}a_{i-1}^{(i)},\frac{N}{m_i}a_i^{(i)}+Nx_i,...,\frac{N}{m_i}a_n^{(i)})>0,\text{ where }N=\prod_{i\in \J} m_i.
\end{equation}
But $\frac{N}{m_i}(a_1^{(i)},... ,a_n^{(i)}) \in \pi(G)$ \hspace{2mm} and $N(x_1,...,x_n)\in \pi(G)$, so after taking the sum, we have that
\begin{equation}
    \label{?}
    l:=(Nx_1+\sum_{i\in \J}  \frac{N}{m_i} a_1^{(i)},...,Nx_n+\sum_{i\in \J} \frac{N}{m_i} a_n^{(i)})\in \pi(G).
\end{equation}

But $l=\sum_{i \in \J}l_i>0$, contradicting the fact that $\pi(G)$ is singular.\qed \par

 \vspace{2mm}
 Our next Claim, allows us to make sure that the group homomorphism (on the $K_0$-level) will be positive. \par

 \textbf{Claim 7:}\hspace{2mm} For $x\in H_i$, the following holds: \par
 \begin{enumerate}[label=\roman*.]
     \item If $\Phi_i(x)\geq 0$, then $\bar \tau_i(x)\geq 0$.
     \item If $\Phi_i(x)\leq 0$, then $\bar \tau_i(x)\leq 0$.
 \end{enumerate}
 \textbf{Proof of Claim 7:} \hspace{2mm} It is enough to prove the first statement for $x\neq 0$. Assume that $\Phi_i(x)=1$. Then there is $k\in \N^*$ and $a\in H_i^{neg}$ such that $kx+a> 0. \text{ Thus } \bar \tau_i(kx)+\bar \tau_i(a)\geq 0$. But, because $a\in H_i^{neg}$, it follows that $\bar \tau_i(a)\leq 0$ by Claim 2, hence $\bar \tau_i(kx)\geq 0$ which implies $\bar \tau_i(x)\geq 0$. \qed \vspace{2mm}

 Define $P_i=\{x\in H_i,\hspace{2mm} \Phi_i(x)\geq 0\} $. Notice that $(H_i,P_i)$ is a countable, totally ordered group (by Claims 4 and 5), hence it is a dimension group. Claim 7 verifies that
 \begin{equation}
 \label{10}
 \{x\in H_i,\hspace{2mm} \bar \tau_i(x)>0\}\subset P_i  \text{ and }\{x\in H_i,\hspace{2mm} \bar \tau_i(x)<0\}\subset - P_i.
 \end{equation}
 Because $\bar \tau_i(\overline{[1_{A_i}]}_0)=\hat{\tau_i}({[1_{A_i}]}_0)=1, \hspace{2mm}\overline{[1_{A_i}]}_0$ is an order unit on this new ordered group. \par
 \vspace{2mm}
 Claim 8 guarantees that each of the totally ordered groups has a unique state, and actually the one needed so that the assumptions of Theorem \ref{lifting_theorem} are satisfied. \par

 \textbf{Claim 8:}\hspace{2mm} $S(H_i,P_i,\overline{[1_{A_i}]}_0)=\{\bar \tau_i\}$. \par
 \textbf{Proof of Claim 8:} \hspace{2mm} Let $\rho \in S(H_i,P_i,\overline{[1_{A_i}]}_0)$ and $x\in H_i$. For better notation set also $u=\overline{[1_{A_i}]}_0$. Assume that $\rho(x)\neq \bar \tau_i(x)$. Assume first that $\rho(x) < \bar \tau_i(x)$. Then $\exists q\in \Q: \rho(x)<q<\bar \tau_i(x)$. Hence we have $$\bar \tau_i(x-qu)>0 \xRightarrow{(\ref{10})} x-qu\geq 0 \Rightarrow x\geq qu \xRightarrow{\rho\in S(H_i)}\rho(x)\geq \rho(qu) \Rightarrow \rho(x)\geq q$$ contradiction. Similarly, we can contradict  the case $\rho(x) > \bar \tau_i(x)$. So $\rho(x)=\bar \tau_i(x)$. Because $x$ is arbitrary, $\rho=\bar \tau_i$. \qed \vspace{2mm}

 By the Effros-Handelman-Shen Theorem \cite[Thm 7.2.6]{Rrdam2000AnIT}, there are $B_i$ unital, AF algebras, such that
 \begin{equation}
     \label{6666}
     (K_0(B_i),K_0(B_i)^+,[1_{B_i}]_0)\cong (H_i,P_i,\overline{[1_{A_i}}]_0)
 \end{equation}
 Because of Claim 8, $B_i$ has a unique trace, $\tau_{B_i}$ and $\hat{\tau}_{B_i}=\bar \tau_i$. Also $B_i$ is $\U$-stable because $H_i\cong H_i\otimes \Q$. Furthermore, we have the following commutative diagram:
 \begin{center}
 \begin{tikzcd}
 K_0(A_i) \arrow {r}{\bar \pi_i} \arrow[swap]{dr}{\hat{\tau_i}} & K_0(B_i) \arrow{d}{\hat{\tau_{B_i}}} \\
 & \R
 \end{tikzcd}
 \end{center}
 where $\bar \pi_i$ is the group homomorphism arising from $\pi_i$ after composing with the order isomorphism implementing \ref{6666}. By \ref{10}, it is positive. Moreover, we have $\bar \pi_i[1_{A_i}]_0=[1_{B_i}]_0$ \par
 Pick $i\in \{1,2,3,...,n\}$. If $A_i \in \mathcal{G}$, $\tau_i$ is faithful, so Theorem \ref{lifting_theorem} applies and hence there is a faithful *-homomorphism $\phi_i:A_i\hookrightarrow B_i$ such that $(\phi_i)_*=\pi_i$. \par
 If $A_i$ is an AF algebra, then by Elliott's Classification Theorem for AF algebras \cite[Thm. 7.3.4]{Rrdam2000AnIT}, $\pi_i$ lifts to a *-homomorphism $\phi_i:A_i\rightarrow B_i$. Note that if $\ker\phi_i$ is non-zero, then it should contain a non-zero projection $p$ (recall that because $A_i$ is an AF-algebra, it has property (SP) which means that every non-zero hereditary $C^*$-subalgebra has a non-zero projection). So $\phi_i(p)=0$, which implies $\pi_i[p]_0=0.$ But this contradicts the fact that $\ker\phi_i \cap K_0(A_i)^+=\{0\}$, so $\phi_i$ is injective for every $i$.
 If we set $B=\oplus_{i=1}^n B_i, \hspace{2mm} \phi=\oplus_{i=1}^n \phi_i,$ then $\phi:A\hookrightarrow B$ is faithful. Also B is an AF algebra and $\phi_*=\pi$. Finally, by Claim 6, $\phi_*(G)$ is singular so there exists $D$ quasidiagonal and $\psi:B\hookrightarrow D$ such that $\psi_*(\phi_*(G))=0$. By composing the two maps, we have that $\psi \circ \phi:A\hookrightarrow D$ satisfies $(\psi \circ \phi)_*(G)=0$. So $A$ has the $K_0$-embedding property.
\end{proof}
\end{proposition}

Actually, Proposition \ref{direct_sums} still holds if, when defining $\mathcal{G}$, we replace the condition that all states should be induced by faithful traces, with the following weaker K-theoretic condition: \par

\emph{For every $G_1\subset G_2\subset K_0(A\otimes \U)$ with $G_1\leq K_0(A\otimes \U)$ singular subgroup and $G_2\subset K_0(A\otimes \U)$ subsemigroup with $G_2\cap -K_0(A\otimes \U)^+=\{0\}$ and $G_2\cap -G_2=G_1$, there exists $\tau\in T(A\otimes \U)$ faithful such that $\hat{\tau}(G_1)=0$ and $\hat{\tau}(z)\geq 0$ for every $z\in G_2$.} \par
However, this condition is very technical and we have not managed to find any interesting $C^*$-algebras that have states that are not induced by any faithful trace, but still satisfy the condition. \par
\vspace{5mm}

Moreover, notice that Proposition \ref{simple} and Proposition \ref{direct_sums} yield another proof of the following part of Theorem \ref{main-theorem}: If $A=\oplus_{i=1}^n A_i$, where each $A_i$ is unital, separable, simple, nuclear, quasidiagonal and satisfies the UCT, $B$ is separable, nuclear, quasidiagonal and satisfies the UCT and
$$\begin{tikzcd}
0 \arrow[r] & A \arrow[r, "\iota"] & E \arrow[r, "\pi"] & B \arrow[r] & 0
\end{tikzcd}$$
is a short exact sequence, then $E$ is quasidiagonal iff stably finite. This proof doesn't use either the results of Section 6, or any of the classification results in \cite{tikuisis2017quasidiagonality} or \cite{MR4228503}. \par

Now we will use the aforementioned results to get new examples of $C^*$-algebras with the $K_0$-embedding property.

\begin{corollary}\label{commutative}
If A is a separable and commutative $C^*$-algebra, then $A$ has the $K_0$-embedding property.
\begin{proof}
First of all, by Proposition \ref{unitization} we may assume that $A$ is unital. Observe that $A$ can be written as an inductive limit $A=\varinjlim C(X_n)$, where the connecting maps are injective and all $X_n$ are separable, compact, Hausdorff with $\dim (X_n)<\infty$. Indeed, let $\{x_n, \hspace{2mm} n\in \N \}$ be a dense subset of $A$, with $x_0=1$ and consider $A_n=C^*(1,x_1,...,x_n)$. For every $n\geq 1$, $A_n$ is unital and commutative so $A_n=C(X_n)$, where $X_n$ is separable, compact and Hausdorff. Moreover, $A=\varinjlim A_n$ where the connecting maps are the inclusions. Finally, by \cite[Prop 1.4]{ng2006note}, $\dim (X_n)<\infty.$ So Remark \ref{inductive_limit} allows us to assume that $A=C(X)$ , where $n=\dim (X)<\infty$. But by \cite[Thm 1.10.16]{engelking1978dimension}, $A$ is locally approximated by algebras $C(X_i)$, where $X_i$ are finite CW complexes with finitely many connected components. By Proposition \ref{direct_sums}, $C(X_i)$ has the $K_0$-embedding Property for every $i$. So, $A$ has the $K_0$-embedding property by Proposition \ref{local_approximations}.
\end{proof}
\end{corollary}

\begin{corollary}\label{direct_sum_crossed_products}
Let $A=\oplus_{i=1}^n A_i$ be a direct sum of $C^*$-algebras and $\sigma:\Z\rightarrow Aut(A)$ be an action such that $\sigma=\sigma_1 \oplus...\oplus \sigma_n$ for some $n\in \N$ and minimal actions $\sigma_i : \Z \rightarrow Aut(A_i)$. Assume that each $A_i$ is separable, nuclear, unital, satisfies the UCT and has a $\sigma_i$-invariant trace. Then $A\rtimes_{\sigma} \Z$ has the $K_0$-embedding property.
\begin{proof}
By \cite[Lemma 2.8.2]{MR911880}, $A\rtimes_{\sigma} \Z= \bigoplus_{i=1}^n A_i \rtimes_{\sigma_i} \Z$. The result follows from Proposition \ref{crossed_products} and Proposition \ref{direct_sums}.
\end{proof}
\end{corollary}

\section{ASH algebras}
Recall that a $C^*$-algebra is called \emph{subhomogeneous} if there exists a positive integer $n$, such that every irreducible representation of $A$ is on a Hilbert space of dimension less or equal to $n$. A $C^*$-algebra is called \emph{approximately subhomogeneous (ASH)} if it is an inductive limit of subhomogeneous $C^*$-algebras. We know that ASH algebras are nuclear, quasidiagonal \cite[Chapter 3.4]{rordam2002classification} and satisfy the UCT (every ASH algebra is an inductive limit of type I $C^*$-algebras, but these algebras satisfy the UCT by \cite[22.3.5]{blackadar1998k}). In \cite{elliott2020decomposition} Elliott, Niu, Santiago and Tikuisis defined an important subclass of ASH algebras.

 \begin{definition}[Def. 2.1, \cite{elliott2020decomposition}]\label{NCCC_1} A \textbf{non-commutative cell complex (NCCC)} is a $C^*$-algebra given by the following recursive definition.
 \begin{enumerate}[label=\roman*.]
     \item A finite dimensional $C^*$-algebra is a NCCC
     \item If $A$ is a NCCC, $n,k\in \N$, $\phi:A\rightarrow C(S^{n-1}, M_k)$ is any unital *-homomorphism, and $\psi: C(D^n, M_k)\rightarrow C(S^{n-1}, M_k)$ is the restriction homomorphism, then the pullback \par
     $A\oplus_{C(S^{n-1}, M_k)} C(D^n, M_k)=\{(a,f)\in A\oplus C(D^n,M_k): \phi(a)=\psi(f)\}$ is also a NCCC.
 \end{enumerate}
 \end{definition}

 We will allow $n=0$. In this case we will use the conventions $D^0=\{pt\}$ and $S^{-1}=\emptyset$. So $\phi$ will be the zero map in this case.\par

 Note that every NCCC is unital. NCCC are a subclass of recursive subhomogeneous (RSH) algebras, which were introduced by C.Phillips in \cite{phillips2007recursive}. The name "non-commutative cell complex" comes from the fact that in the commutative case, this is equivalent to $A$ being isomorphic to $C(X)$, where $X$ is a (finite) cell complex. \par
 The following definition is on the style of \cite[Def 1.2]{phillips2007recursive}, but for NCCC:

 \begin{definition}\label{NCCC_2}
 From the previous definition it is clear that every NCCC is of the form \par
\begin{equation}
     \label{6.1}
     A=[...[F_0\oplus_{C(S^{{n_1}-1},M_{r_1})} C(D^{n_1},M_{r_1})]\oplus_{C(S^{{n_2}-1},M_{r_2})}C(D^{n_2},M_{r_2})...]\oplus_{C(S^{{n_l}-1},M_{r_l})}C(D^{n_l},M_{r_l})
 \end{equation}
 with $F_0$ finite dimensional. Note that by definition it can be $F_0=0$. Set $A_0=F_0$ and $A_i$ to be the $i-th$ pullback, for every $i>0$. We will denote with $\phi_i=\phi_i^{(A)}:A_i\rightarrow C(S^{n_{i+1}-1},M_{r_i}), \psi_i=\psi_i^{(A)}:C(D^{n_i},M_{r_i})\rightarrow C(S^{n_i-1},M_{r_i}) $  the *-homomorphisms defining the pullback. Recall that $\phi_i(a)=\psi_{i+1}(f)$ for every $(a,f)\in A_{i+1}$.\par An expression of this type will be called a \textbf{decomposition} of $A$. Note that such a decomposition is not unique. \par
 Associated with this decomposition are:
 \begin{enumerate}[label=\roman*.]
     \item its \textbf{length} $l$.
     \item its \textbf{base spaces} $X_0$,$X_1,...,X_l$ (where $X_0$ is a disjoint union of singletons and $X_i=D^{n_i}$, for $i\geq 1$) and \textbf{total space} $X:=\coprod X_k$.
     \item its \textbf{$i-th$ stage algebra} $A_i (i=0,1,2,...l)$.
     \item its \textbf{(topological) dimension} $\dim (A)=\max_k\dim(X_k)$.
     \item its \textbf{standard representation} $\sigma=\sigma^{(A)}:A\rightarrow F_0\oplus (\oplus_{i=1}^l C(D^{n_i}, M_{r_i}))$ defined by forgetting the restriction to a subalgebra in each of the fibered products in the decomposition.
     \item the associated \textbf{evaluation maps} $ev_x:A\rightarrow M_{r_i}$ for $x\in X$.
     \item the \textbf{rank function} $\rank:P_{\infty}(A)\rightarrow C(X,\N)$ via the natural definition (recall that $X$ is the total space).
 \end{enumerate}
 \end{definition}
Before going forward, let's clarify some notation. For the rest of the section, unless clearly stated otherwise, $A$ will be a NCCC of length $l$ with decomposition as in (6.1). Let $p\in P_{\infty}(A)$. Because of the standard representation, we can view $p$ inside $P_{\infty}(F_0\oplus (\oplus_{i=1}^l C(D^{n_i}, M_{r_i})))$, so we may write $p=(p_0,...,p_l)$, where $p_0\in P_{\infty}(F_0)$ and $p_i\in P_{\infty}(C(D^{n_i}))$ for every $i>0$. Moreover, for every $i>0$ we will write $\rank p_i$ instead of $\rank p_i(x)$, because $D^{n_i}$ is connected, so rank is constant on each $D^{n_i}$. \par  When denoting $y=[p]_0-[q]_0$ we will mean that $p=(p_0,...,p_l)$,$q=(q_0,...,q_l)\in P_{\infty}(A)$, and we will also write $\bar p_s=(p_0,...,p_s), \bar q_s=(q_0,...,q_s), y_s=[\bar p_s]_0-[\bar q_s]_0$, for $s=0,1,2,...,l$. \par

 The reason why we chose to work with NCCC is the fact that every unital separable, subhomogeneous algebra is locally approximated by NCCC.
 \begin{proposition} [Thm. C.,  \cite{elliott2020decomposition}]\label{ASH_approximation} Let $A$ be a unital separable subhomogeneous algebra. Then $A$ is locally approximated by NCCC.
 \end{proposition}

Let $A$ be a NCCC. Our goal is to show that $A$ has the $K_0$-embedding Property (see Proposition \ref{NCCC_K_0_embedding_property}). We first present a sketch of the proof. Fix $y=[p]_0-[q]_0$, where $p=(p_0,...,p_l),q=(q_0,...,q_l)\in P_{\infty}(A)$. If there exist $x_1,x_2\in X$ with $\rank p(x_1)>\rank q(x_1)$ and $\rank p(x_2)<\rank q(x_2)$, then $\sigma_*(y)$ is non-zero singular. \par
If $\rank p(x)>\rank q(x)$ for every $x$, then by \cite[Prop 4.3]{phillips2007recursive} we have that there exists $M>0$ such that $My>0$ in $K_0(A)$. \par
Assume that $\rank p(x)\geq \rank q(x)$ for every $x$, but there exists $x_0$ such that $\rank p(x_0)=\rank q(x_0)$. Then it follows that $\sigma_*(x)\geq 0$, but the previous conclusion is not true anymore, as we can't use \cite[Prop 4.3]{phillips2007recursive}. For example, consider $A=A_1\oplus A_2$, where $A_i$ is a NCCC for $i=1,2$ (it is not difficult to see that $A$ is a NCCC), $p=(p_1,p_2)$ and $ q=(q_1,q_2)\in P_{\infty}(A)$, where $\rank p_1(x)>\rank q_1(x)$ for every $x$, $[p_2]_0-[q_2]_0\in K_0(A_2)$ is not a torsion element and $[p_2]_0-[q_2]_0\in \ker \sigma_*^{(A_2)}.$ Then $y=[p]_0-[q]_0\in K_0(A)$ is singular but $\sigma_*^{(A)}(y)>0$.
In order to bypass this problem, for every $y=[p]_0-[q]_0\in K_0(A)$, with $\sigma_*^{(A)}(y)\geq 0$, we will construct a *-homomorphism $h_y:A\rightarrow D_y(A)$, where $D_y(A)$ is an AF algebra with the property that if $(h_y)_*(y)\geq 0$, then there exists $M>0$ such that $My\geq 0$. This map will be constructed in two steps. First, we will show that we can construct a NCCC $R_y(A)$, which can be obtained from $A$ after deleting all the coordinates $j>0$ such that $\rank p_j>\rank q_j$. To show this, our main tool is Lemma \ref{main_lemma} while the construction is made right after it. Then we will construct (in Proposition \ref{main_technical_result}) a *-homomorphism $\Psi_y:A\rightarrow R_y(A)$, with the property that if $(\Psi_y)_*(y)\geq 0$, then there exists $M>0$ with $My\geq 0$. The key why the latter property will hold is Lemma \ref{lemma_2_gen}. The second step is to construct the *-homomorphism to the AF algebra. This will be achieved by showing the existence of a decomposition $R_y(A)=L_1\oplus L_2 $, such that the first coordinate of $(\Psi_y)_*(y)\in K_0(R_y(A))\cong K_0(L_1)\oplus K_0(L_2)$ is positive while the second one is an infinitesimal (Lemma \ref{empty_case} and Proposition \ref{main_technical_result}). Then the desired *-homomorphism will be the composition of the projection to the second coordinate with the embedding from Corollary \ref{AF-embedding}. Moreover, we will make all such $\Psi_y$ (recall that $y$ runs through the elements of $K_0(A)$ such that $\rank p(x)\geq \rank q(x)$ for every $x$) to belong to a finite set of *-homomorphisms. So, if we take their (finite) direct sum and then add the standard representation, we get a faithful *-homomorphism $h:A\rightarrow E(A)$. It won't be difficult to show that $h$ sends singular elements to singular elements. Notice that $E(A)\in \mathcal{O}$. So Proposition \ref{direct_sums} will give us that $A$ has the $K_0$-embedding Property.
  \vspace{5mm}

 \vspace{5mm}

Before starting the detailed proof, we need to introduce some more notation.
 \begin{definition}
 Let $A$ be a NCCC of length $l$ and $y=[p]_0-[q]_0\in K_0(A)$, where $p=(p_0,...,p_l)$,$q=(q_0,...,q_l)\in P_{\infty}(A)$. We will say that $y$ is \textbf{almost positive} if $\rank p(x)\geq \rank q(x)$ for every $x$ on the total space of $A$.
 \end{definition}
It is known what the trace simplex $T(A)$ looks like.
 \begin{lemma}[Cor. 2.5, \cite{elliott2017classification}]\label{trace_simplex}
 Any trace $\tau\in T(A)$ is of the form
 $$\tau(f_0,...,f_l)=a_0 \tau_0(f_0)+\sum_{i=1}^l a_i \int_{D^{n_i}\backslash S^{n_i-1}}tr(f_i(x))d\mu_i(x),$$
 where $\tau_0$ is a trace in $F_0$, $\mu_i$ is a probability measure in $D^{n_i}\backslash S^{n_i-1}, a_i\in [0,1]$ and $a_0+a_1+...+a_l=1$.

 \end{lemma} 

Recall that if $X$ is a contractible topological space, then $(K_0(C(X)),K_0(C(X))^+)\cong (\Z,\Z^+)$. More specifically the order isomorphism is $f\mapsto \rank f$ (see for instance \cite[Example 3.3.6]{Rrdam2000AnIT}. But $D^{n_i}$ is contractible so every $i>0$. So, $K_0(F_0\oplus (\oplus_{i=1}^l C(D^{n_i}, M_{r_i})))$ has no non-zero infinitesimals and equality on the $K_0$-group coincides with equality of the ranks. This observation allows us to deduce the following Lemma.

 \begin{lemma}\label{infinitesimals_NCCC}
 The group of infinitesimals of $A$ is
 $$\Inf(K_0(A))=\ker(\sigma_*)=\{[f]_0-[g]_0 \hspace{2mm} \vert \hspace{2mm} \rank f(x)=\rank g(x) \text{ for every x in the total space }\}.$$
 \begin{proof}
 Let $y=[p]_0-[q]_0\in \Inf(K_0(A))$. By Lemma \ref{infinitesimals_after_homomorphisms}, $\sigma_*(y)\in \Inf K_0(F_0\oplus (\oplus_{i=1}^l C(D^{n_i}, M_{r_i})))$. Hence $\sigma_*(y)=0$, by the comments above. If $y\in \ker(\sigma_*)$, then, again by the comments above, we have $\rank p(x)=\rank q(x)$ for every $x$ in the total space. Finally, assume that $\rank p(x)=\rank q(x)$ for every $x$. Then by Lemma \ref{trace_simplex}, $\hat{\tau}(y)=0$ for every $\tau \in T(A)$. Hence $y\in \Inf(K_0(A))$.
 \end{proof}
\end{lemma}
\vspace{5mm}
 Let $y=[p]_0-[q]_0\in K_0(A)$. Define  $$\W=\W_A(y):=\{i\geq 1: \rank p_i>\rank q_i\}.$$

 \begin{lemma}\label{main_lemma}
 Let $y=[p]_0-[q]_0$ be almost positive with $\rank p_l=\rank q_l$. Assume that $\W_A(y)\neq \emptyset$ and set $j:=\max\W_A(y)$. Then
 $\phi_i(0,...,0,a,0,...,0)=0 \text{ for every }i\geq j \text{ and }a\in \ker\psi_j$.

 \begin{proof}
 Let $y=[p]_0-[q]_0$ be almost positive with $\rank p_l=\rank q_l$. Assume that $\W_A(y)\neq \emptyset$ and set $j:=\max\W_A(y)$. By assumption, we have that $0<j<l$. We have $A_{j+1}=[A_{j-1}\oplus_{C(S^{{n_j}-1},M_{r_j})} C(D^{n_j},M_{r_j})]\oplus_{C(S^{{n_{j+1}}-1},M_{r_{j+1}})} C(D^{n_{j+1}},M_{r_{j+1}})$. Pick $s\in S^{n_{j+1}-1}$ and denote $\widetilde{\phi_j}:=ev_s \circ \phi_j: A_j\rightarrow M_{r_{j+1}}$. Then by \cite[Remark 4.7]{elliott2017classification}, we get that $\widetilde{\phi_j}$ factors (up to unitary equivalence which we may ignore because it doesn't affect the $K_0$ map) through the direct sum via the following commutative diagram: \par
 \begin{tikzcd}
A_j \arrow[d, "\sigma"] \arrow[r, "\widetilde{\phi_j}"]         & M_{r_{j+1}}                                  \\
{A_{j-1}\oplus C(D^{n_j},M_{r_j})} \arrow[r, "\phi_B \oplus \phi_D"] & M_{q}\oplus M_{w} \arrow[u, "\iota"]
\end{tikzcd}
 \par
 where $\sigma$ is the standard representation, $\phi_B,\phi_D$ are *-homomorphisms (not necessarily unital; they could be even zero), $q+w=r_{j+1}$ and $\iota$ is the diagonal embedding. By maximality of $j$, we have that $\rank p_{j+1}=\rank q_{j+1}$. But $D^{n_{j+1}}$ is contractible, so $[p_{j+1}]_0=[q_{j+1}]_0$ in $K_0(C(D^{n_{j+1}}))$. Notice that $(\phi_j)_*(y_j)=(\psi_{j+1})_*([p_{j+1}]_0-[q_{j+1}]_0)=0$. So,
 \begin{equation}
     \label{***}
     (\widetilde{\phi_j})_*(y_j)=0.
 \end{equation}
  Let $w=\rank p_j-\rank q_j>0$ and consider $(\phi_D)_*:\Z\rightarrow \Z$. By the commutativity of the diagram, as well as (\ref{***}), we can see that $(\phi_D)_*(w)=0$. Hence $(\phi_D)_*=0$. Thus it has to be zero. This means that $\phi_D=0$. If now $a\in \ker\psi_j$, then from all the aforementioned we deduce that $$\widetilde{\phi_j}(0,a)=0\Rightarrow \phi_j(0,a)(s)=0.$$
 But $s$ is arbitrary, so $\phi(0,a)=0$. \par

 Let now $l>i>j$ and assume that the result holds for up to $i-1\geq j$. We will show it for $i$. By repeating the arguments for the base case we get the following commutative diagram (after fixing $s\in S^{n_{i+1}-1}$) \par
  \begin{tikzcd}
A_i \arrow[d, "\sigma"] \arrow[r, "\widetilde{\phi_i}"]         & M_{r_{i+1}}                                  \\
{A_{j-1}\oplus C(D^{n_j},M_{r_j})}\oplus...\oplus C(D^{n_i}, M_{r_i}) \arrow[r, "h_{j-1} \oplus...\oplus h_{i}"] & M_{q(j-1)}\oplus...\oplus M_{q(i)} \arrow[u, "\iota"]
\end{tikzcd}
\par
where $\widetilde{\phi_i}=ev_y \circ \phi_i$. Once again $\sigma$ is the standard representation, $\iota$ the diagonal embedding and $h_{j-1},...,h_{i}$ are *-homomorphisms. \par
By the same arguments we get $\widetilde{\phi_i}(y_i)=0$ and $h_j=0$. By the commutativity of the diagram and the fact that $y$ is arbitrary, $\phi_i(0,...,0,a,0,...,0)=0$, for every $a\in \ker(\psi_j)$ as desired. Note that the fact that $(0,0,...,0,a,0,...,0)$ is a well-defined element in $A_i$ is guaranteed by the inductive hypothesis.
 \end{proof}
 \end{lemma}

 Let $y=[p]_0-[q]_0\in K_0(A)$ be almost positive and $l>j=\max\W_A(y)$. Note that we still assume $\W_A(y)\neq \emptyset$. Assume that $j<l$.
 Define $$\bar \phi_j:A_{j-1}\rightarrow C(S^{n_{j+1}-1}, M_{r_{j+1}})$$
 via  $(\bar \phi_j)(a)=\phi_j(a,b)$ for every $a\in A_{j-1}$ and  some (all) $b\in C(D^{n_j}, M_{r_j}) \text{ such that }(a,b)\in A_j$. \par
 Notice that $\bar \phi_j$ is a well-defined, unital *-homomorphism (or zero if $\phi_j=0$).  Indeed, if $(a,b_1),(a,b_2)\in A_j$, then $$\phi_j(a,b_1)-\phi_j(a,b_2)=\phi_j(0,b_1-b_2)=0$$ by Lemma \ref{main_lemma}. Hence, we can define the pullback $$\overline{A_{j+1}}:=A_{j-1}\oplus_{ C(S^{n_{j+1}-1}, M_{r_{j+1}})} C(D^{n_{j+1}}, M_{r_{j+1}})$$
 Note that the maps defining the pullback are $\bar \phi_j$ and $\psi_{j+1}.$
 Set $$\Phi_{j+1}: A_{j+1}\rightarrow \overline{A_{j+1}}$$ via $\Phi_{j+1}(f,g,h)=(f,h),$ where $f\in A_{j-1}, g\in C(D^{n_j}, M_{r_j}), h\in C(D^{n_{j+1}}, M_{r_{j+1}}).$\par Note that $\bar \phi_j(f)=\phi_j(f,g)=\psi_{j+1}(h).$ Thus $(f,h)\in \overline{A_{j+1}}.$
 So, $\Phi_{j+1}$ is a well-defined unital *-homomorphism. \par
Actually, we can generalize this construction: \par
If $i>j$, then we can inductively define the pullback
\begin{equation}
\label{6.3}
\overline{A_{ij}}:=[A_{j-1}\oplus_{ C(S^{n_{j+1}-1}, M_{r_{j+1}})} C(D^{n_{j+1}}, M_{r_{j+1}})]\oplus... \oplus_{ C(S^{n_i-1}, M_{r_i)}} C(D^{n_i}, M_{r_i})
\end{equation}
by using the map
$$\bar \phi_{ij}:[A_{j-1}\oplus_{ C(S^{n_{j+1}-1}, M_{r_{j+1}})} C(D^{n_{j+1}}, M_{r_{j+1}})]\oplus... \oplus_{ C(S^{n_{i-1}-1}, M_{r_{i-1})}} C(D^{n_{i-1}}, M_{r_{i-1}})\rightarrow C(S^{{n_i}-1},M_{r_i})$$
with $$\bar \phi_{ij}(\bar f_{j-1},f_{j+1},...,f_{i-1})=\phi_i(\bar f_{j-1},f_j,f_{j+1},...,f_{i-1})$$
for some(any) $f_j\in C(D^{n_j, M_{r_j}})$ such that the right hand side is well-defined. Notice that Lemma \ref{main_lemma} guarantees that the map is well defined. Furthermore set
\begin{equation}
    \label{6.4}
     \Phi_{ij}: A_i\rightarrow \overline{A_{ij}}
\end{equation}

where the map takes an element of $A_i$ and "removes" it's j-th component. \par
\begin{lemma}\label{lemma_2_gen}
If $(\Phi_{ij})_*(y_i)\geq 0$, then there is $M\in \N^* \text{ such that } My_i>0$ in $K_0(A_i)$.
 \begin{proof}
Suppose that $(\Phi_{ij})_*(y_{i})\geq 0$. Also, (after replacing $p,q$ with $p\oplus 1_s$ and $q\oplus 1_s$ for large enough $s$; note that we can do this because $\rank ((p\oplus 1_s)(x))-\rank ((q\oplus 1_s)(x))=\rank p(x)-\rank q(x)$ for every $x$), we may assume that there exists a partial isometry $v=(\bar v_{j-1},v_{j+1},...,v_{i})\in M_{\infty}(\overline{A_{ij}})$ such that $v^*v=(\bar q_{j-1},q_{j+1},...,q_{i})$ and $vv^*\leq (\bar p_{j-1},p_{j+1},...,p_{i})$. Recall that $\bar v_{j-1}\in M_{\infty}(A_{j-1})$ and $v_{k}\in M_{\infty}(C(D^{n_{k}}))$ for every $k=j+1,...,i$. Because $j=\max \W_A(y)$, we have $\rank p_j>\rank q_j$. Again, we may assume that $\rank p_j-\rank q_j>\frac{\dim (A)-1}{2}$ (if this is not true, then we can take direct sums $p_j\oplus p_j\oplus...\oplus p_j$ and $q_j\oplus q_j\oplus...\oplus q_j$ as large needed to achieve this). By using \cite[Prop 4.2]{phillips2007recursive} for $p_j,q_j, S^{n_j-1}\subset D^{n_j}$ and the partial isometry $\phi_{j-1}(\bar v_{j-1})\in M_{\infty}(C(S^{n_j-1}))$, we get that there exists a partial isometry $v_j\in M_{\infty}(C(D^{n_j}))$, such that $\widetilde{v}:=(\bar v_{j-1},v_j)\in M_{\infty}(A_j)$ and also satisfies
$$\widetilde{v}\widetilde{v}^*<\bar p_j \text{ and } \widetilde{v}^*\widetilde{v}=\bar q_j.$$
Let $w:=(\bar v_{j-1},v_j,v_{j+1},...,v_i)\in M_{\infty}(A_{i})$. Then
$$ww^*<\bar p_{i} \text{ and } w^*w=\bar q_{i}$$
so $y_{i}>0$.

\end{proof}
\end{lemma}

We are left to deal with the case $\W_A(y)=\emptyset$. On our next lemma, we will show that this case (under mild extra assumptions) leads to a direct sum decomposition with nice properties. But first we need to recall some notation from \cite[Ch. 1.5]{rordam2002classification}. \par
Let $E$ be a unital and stably finite $C^*$-algebra. Let $I\leq K_0(E)$ and $I^+=K_0(E)^+\cap I$. We say that $(I,I^+)$ is an \textbf{ideal} in $(K_0(E),K_0(E)^+)$ if:
\begin{enumerate}[label=\roman*.]
    \item $I=I^+-I^+$ and
    \item for all $x,y\in K_0(E)$, if $0\leq x\leq y$ and $y\in I^+$, then $x\in I^+$.
\end{enumerate}
If $S\subset K_0(E)^+$ is a subsemigroup and $(I,I^+)$ is an ideal, we say that $(I,I^+)$ is \textbf{generated by S} if $I^+=\{a\in K_0(E)^+ \hspace{3mm} | \hspace{3mm} 0 \leq a\leq b \text{ for some }b\in S\}.$ \par

\begin{lemma} \label{basic_order_ideal_lemma}
Let $F$ be a finite dimensional $C^*$-algebra and $\Gamma \subset K_0(F)^+$ be a subsemigroup. Assume that $(I,I^+)$ is the ideal of $(K_0(F),K_0(F)^+)$ generated by $\Gamma$. Then there exist $a\in \Gamma$ such that $(I,I^+)$ is generated by $a$.
\begin{proof}
First of all, $F$ is a direct sum of matrix algebras, so $(K_0(F),K_0(F)^+)\cong (\Z^n,\Z_+^n)$ for some $n$. Note that $I$ must be of the form $$I=\{(b_1,...,b_n)\vert b_j=0 \text{ for every }j\in J\}$$ with $J$ being some subset of $\{1,2,...,n\}$. Moreover, for every $j\notin J$, there is $a^{(j)}=(a^{(j)}_1,...,a^{(j)}_n)\in \Gamma$ with the property that $a_j^{(j)}>0.$ Take $a=\sum_{j\notin J} a^{(j)}$ and observe that this is our generator.
\end{proof}
\end{lemma}

\begin{lemma}\label{empty_case}
Let $A$ be a NCCC with length $l>0$ and decomposition as in (6.1). Let also $y=[p]_0-[q]_0$ almost positive, such that $\W_A(y)=\emptyset$ and assume that there exists $x_0\in X_0$ such that $\rank p_0(x)>\rank q_0(x)$. Then there exist $F_1, F_2$ finite dimensional $C^*$-algebras with $F_0=F_1\oplus F_2$  that satisfy the following properties:
\begin{enumerate}[label=\roman*.]
    \item $A=F_1\oplus B$, where $B=[...[F_2\oplus_{C(S^{{n_1}-1},M_{r_1})} C(D^{n_1},M_{r_1})]\oplus_{C(S^{{n_2}-1},M_{r_2})}C(D^{n_2},M_{r_1})...]\oplus_{C(S^{{n_l}-1},M_{r_l})}C(D^{n_l},M_{r_l})$ is a NCCC.
    \item  $\Gamma_2:=\{a\in K_0(F_2)^+: (a,0,...0)\in \sigma_*(K_0(B))\}=\{0\}.$

\end{enumerate}
\begin{proof}
Define
\begin{equation}
\label{*******}
\Gamma:=\{a\in K_0(F_0)^+: (a,0,...0)\in \sigma_*(K_0(A))\}
\end{equation}
Note that $\Gamma$ is a semigroup and $\Gamma\neq \{0\}$ by assumption.
Let $(I,I^+)$ be the ideal generated by $\Gamma$. By Lemma \ref{basic_order_ideal_lemma}, there exists $a\in \Gamma$ such that $(I,I^+)$ is generated (as an ideal) by $a$. 
\par
It is known (see \cite[Prop. 1.5.3]{rordam2002classification}) that there exists $F_1$ ideal of $F_0 \text{ such that } (K_0(F_1),K_0(F_1)^+)\cong (I,I^+)$. But in finite dimensional $C^*$-algebras, ideals are always summands, so there exists $F_2$ such that
\begin{equation}
\label{6.5}
F_0=F_1\oplus F_2.
\end{equation}
Let $z=[f]_0-[g]_0$, where $f=(f_0,...,f_l)$ and $g=(g_0,...,g_l)\in P_{\infty}(A)$ such that $\sigma_*(z)=(a,0,...,0)$ (such $z$ exists because $a\in \Gamma)$. Then
\begin{equation}
    \label{6.6}
    \rank (f_i)=\rank (g_i) \text{ for every }i\geq 1.
\end{equation}
We have that $[f_0]_0-[g_0]_0=a=[h]_0$ for some $h\in P_{\infty}(F_0)$. It follows that

$$[f_0]_0=[h\oplus g_0]_0\Rightarrow[\phi_0(f_0)]_0=[\phi_0(h\oplus g_0)]_0\Rightarrow \rank  \phi_0(f_0)=\rank  \phi_0(h\oplus g_0)\Rightarrow$$ $$\Rightarrow \rank (f_1)=\rank (\phi_0(h))+\rank (g_1)\xRightarrow{(6.7)}\phi_0(h)=0\Rightarrow \phi_0(1_{F_1})=0\Rightarrow (1_{F_1},0)\in A_1$$
${\text{ where we used the facts that } [h]=a \text{ and }(I,I^+) \text{ is generated by }a}$ to get the second to last equation. \par
We will use induction to show that $\phi_i(1_{F_1},0,...,0)=0$ for every $i$. We have already shown it for $i=0$. Assume that it holds for $i-1$. We will show it for $i$. \par
We have:
$$[(f_0,...,f_i)]_0-[(h\oplus g_0,g_1,...,g_i)]_0\in \Inf K_0(A) \Rightarrow$$
$$[\phi_i(f_0,...,f_i)]_0-[\phi_i(h\oplus g_0,g_1,...,g_i)]_0\in \Inf K_0(C(S^{n_{i+1}-1}, M_{r_{i+1}})) \Rightarrow$$
$$ \rank  \phi_i(f_0,f_1,..,f_i)=\rank  \phi_i(h\oplus g_0,g_1,...,g_i)\Rightarrow$$ $$\rank (f_{i+1})=\rank (\phi_0(h,0,...,0))+\rank (g_{i+1})\Rightarrow \phi_i(h,0,...,0)=0\Rightarrow$$ $$ \phi_i(1_{F_1},0,...,0)=0.$$
We used Lemma \ref{infinitesimals_after_homomorphisms} and the fact that if $x=[p]-[q]\in \Inf K_0(C(X))$, then $\rank(p)=\rank(q)$. This follows immediately from (2.6). Finally, note that because of the induction hypothesis everything is well-defined. \par
So, we have
\begin{equation}
    \label{6.7}
    A \cong F_1\oplus B
\end{equation}
where $B=[...[F_2\oplus_{C(S^{{n_1}-1},M_{r_1})} C(D^{n_1},M_{r_1})]\oplus_{C(S^{{n_2}-1},M_{r_2})}C(D^{n_2},M_{r_1})...]\oplus_{C(S^{{n_l}-1},M_{r_l})}C(D^{n_l},M_{r_l})$ is a NCCC. The isomorphism is via the natural map and the relation $\phi_i(1_{F_1},0,...,0)=0$ for every $i$ guarantees that everything is well-defined. Moreover, let
$$\Gamma_2=\{a\in K_0(F_2)^+: (a,0,...0)\in \sigma_*(K_0(B))\}.$$
We will show that $\Gamma_2=\{0\}.$ Indeed if $a\in K_0(F_2)^+\backslash \{0\}$ such that $\sigma_*^{(B)}(x)=(a,0,...,0)$ for some $x\in K_0(B)$, then
$$\sigma_*^{(A)}(0,x)=((0,a),0,...,0)\Rightarrow (0,a)\in \Gamma \subset I^{+} \subset K_0(F_1)$$
which clearly cannot happen.

\end{proof}
\end{lemma}

\begin{proposition}\label{main_technical_result}
Let $A$ be a NCCC and $y=[p]_0-[q]_0\in K_0(A)$ almost positive. There exists $R=R_y(A)$ a NCCC and $\Psi=\Psi_y^{(A)}:A\rightarrow R_y(A)$ unital *-homomorphism that has the following properties:
\begin{enumerate}[label=\roman*.]
    \item If $(\Psi_y^{(A)})_*(y)\geq 0$, then there exists $M>0$ such that $My\geq 0.$
    \item $R_y(A)=F_1^{y,A}\oplus B_y(A)$, where $F_1^{y,A}$ is a finite dimensional $C^*$-algebra and $B_y(A)$ is a NCCC. Moreover, if $(\Psi_y^{(A)})_*(y)=(y^{(1)},y^{(2)})$, then $y^{(1)}\geq 0$ and $y^{(2)}\in \Inf K_0(B_y(A))$. (It has to be noted that we allow any of the summands to be 0)
    \item For given $A$,
    $$ \vert \{B_y(A) \hspace{2mm} \vert \hspace{2mm} y \text{ is almost positive }\} \vert\leq 2^n \vert \{R_y(A) \hspace{2mm} \vert \hspace{2mm} y \text{ is almost positive }\} \vert =$$
    $$=2^n\vert \{\Psi_y^{(A)} \hspace{2mm} \vert \hspace{2mm} y \text{ is almost positive }\} \vert  <\infty,$$ where $n$ is the number of disjoint singletons in $X_0$. 
\end{enumerate}
\begin{proof}
We will define $R_y(A),\Psi_y^{(A)}$ with induction on the length of $A$. \par
If the length of $A$ is zero, then define $R_y(A)=A$ and $\Psi_y^{(A)}=id$ for every $y$. Properties (i)-(iii) hold trivially. \par
Suppose that we have defined $R_y(L), \Psi_y^{(L)}$ for every $L$ NCCC with length less or equal than $l-1$ and for every $y\in K_0(L)$ almost positive. Let $A$ be a NCCC with length $l$ and $y=[p]_0-[q]_0\in K_0(A)$ almost positive. We have four cases:\par

\textbf{Case 1:} $\rank p_l>\rank q_l$. \par
Denote $\pi:A\rightarrow A_{l-1}$ to be the projection to the $l-1$ th stage algebra. Set $$R_y(A)=R_{\pi_*(y)}(A_{l-1}) \text{ and } \Psi_y^{(A)}=\Psi_{\pi_*(y)}^{A_{l-1}}\circ \pi$$
Note that the length of $A_{l-1}$ is $l-1$, so everything is well-defined by induction hypothesis. \par
Assume that $(\Psi_y^{(A)})_*(y)\geq 0$. This means that $$(\Psi_{\pi_*(y)}^{A_{l-1}}\circ \pi)_*(y)\geq 0.$$
By induction hypothesis there exists $N\in \N$ such that $$N\pi_*(y)\geq 0.$$

Let now $M\in\N: N\vert M$ and $\dim (A)<M$. After replacing $p,q$ with $p\oplus 1_s$ and $q\oplus 1_s$ for large enough $s$, we may assume that $$ \bar p_{l-1}^M:=\bar p_{l-1}\oplus... \oplus \bar p_{l-1} \succeq \bar q_{l-1}\oplus... \oplus \bar q_{l-1}:=\bar q_{l-1}^M \text{ (each summand is taken M times) }$$
So there is a partial isometry $v\in M_{\infty}(A_{l-1})$ such that
$$vv^*=\bar q_{l-1}^M \text { and } v^*v\leq \bar p_{l-1}^M$$
Similarly define $$p_l^M:=p_l\oplus... \oplus p_l \text{ and } q_l\oplus... \oplus q_l:=q_l^M$$
Because $\rank p_l>\rank q_l$, $\rank p_l^M-\rank q_l^M\geq M>\frac{\dim (A)-1}{2}$. So the hypothesis of \cite[Prop 4.2]{phillips2007recursive} is satisfied for $p_l^M, q_l^M, S^{n_l-1}\subset D^{n_l}$ and the partial isometry $\phi_{l-1}(v)$. Thus, $\phi_{l-1}(v)$ can be extended to a partial isometry $w$ on $M_{\infty}(C(D^{n_l}))$ such that
$$ww^*=q_l\oplus...\oplus q_l \text{ and } w^*w<p_l\oplus...\oplus p_l$$ Hence, by considering the partial isometry $(v,w)\in M_{\infty}(A)$, we get that $My> 0$, so (i) is satisfied. \par
\vspace{5mm}
\textbf{Case 2:} $\rank p(x)=\rank q(x)$ for every $x$ in the total space. \par
Set $$R_y(A)=A \text{ and } \Psi_y^{(A)}=id.$$
Property (i) holds trivially and Property (ii) holds for $R_y(A)=0\oplus A$ because $y\in \Inf(K_0(A))$ by Lemma \ref{infinitesimals_NCCC}. \par
\vspace{5mm}
\textbf{Case 3:} $\W_A(y)=\emptyset$ and there is $x_0\in X_0$ such that $\rank p_0(x_0)>\rank q_0(x_0)$.\par
Set $$R_y(A)=A \text{ and } \Psi_y^{(A)}=id.$$
By Lemma \ref{empty_case}, $A=F_1\oplus B$, where $F_1$ is a finite dimensional $C^*$-algebra and $B$ is a NCCC. Let $y=(y^{(1)},y^{(2)})$. Because $y$ is almost positive, $y^{(1)}\geq 0.$ Because $\Gamma_2=\{0\}$ ($\Gamma_2$ is as defined in Lemma \ref{empty_case}), it follows that $\sigma_*^{(B)}(y^{(2)})=0$. Thus, Lemma \ref{infinitesimals_NCCC} yields that $y^{(2)}\in \Inf(K_0(B))$. So, (ii) holds. Moreover, (i) holds trivially. \par
\textbf{Case 4:} $\W_A(y)\neq \emptyset$ and $\rank p_l=\rank q_l$. \par
Let $j=\max\W_A(y)$. By assumption, $j<l$. Recall that in (6.4) we defined a map
$$\Phi_{lj}:A\rightarrow \overline{A_{lj}}$$
where $\overline{A_{lj}}$, is as defined in (6.3) and has length $l-1$. Set
$$R_y(A)=\overline{A_{lj}} \text{ and } \Psi_y^{(A)}=\Psi_{(\Phi_{lj})_*(y)}^{(\overline{A_{lj}})}\circ \Phi_{lj}.$$
By induction hypothesis, everything is well-defined. Assume that $(\Psi_y^{(A)})_*(y)\geq 0$. This means that
$$(\Psi_{(\Phi_{lj})_*(y)}^{(\overline{A_{lj}})}\circ \Phi_{lj})_*(y)\geq 0.$$
By induction hypothesis there is $N \in \N^*$ such that $$M(\Phi_{lj})_*(y)\geq 0.$$
Thus by Lemma \ref{lemma_2_gen}, there exists $M'>0$ such that $M'y>0$, as desired. So (i) holds. \par
We will now show that (ii) holds for Cases 1 and 4. \par
In both cases, notice that on the inductive step the cardinality of $\W_A(y)$ decreases by 1. Moreover, the inductive step preserves almost positivity, and as long as the cardinality remains non-zero, the presence in one of these two cases. So, when we reach $R_y(A)$, the cardinality should become zero, which means that we now lie on one of the other two cases. But Property (ii) is about $R_y(A)$ and we have already showed it for cases  2 and 3. So, it holds for Cases 1 and 4 as well.\par

We are left to show Property (iii). \par
Note that for every $y\in K_0(A)$ almost positive, $R_y(A)$ is formed from $A$ after "deleting" the coordinates of the set $\W_A(y)$. But this is done uniquely, so $R_y(A)$ depends only on the elements of $\W_A(y)$. Observe that $\Psi_y^{(A)}$ is completely determined by $R_y(A)$. Note that $R_y(A)=F_1^{y,A}\oplus B_y(A)$ as in Lemma \ref{empty_case}, if $y_0>0$, while $R_y(A)=0\oplus R_y(A)$ if $y_0=0$. By looking at the proof of Lemma \ref{empty_case}, $F_1$ depends only on the ideal $(I,I^+)$. So, $B_y(A)$ is completely determined by the elements of $\W_A(y)$ plus what the ideal $(I,I^+)$ is. But $(K_0(F_0),K_0(F_0)^+)\cong (\Z^n, \Z^n_+)$, where $n$ is the number of disjoint singletons in $X_0$. So, we have $2^n$ choices for $(I,I^+)$ (see the proof of Lemma \ref{basic_order_ideal_lemma}). Hence
$$ \vert \{B_y(A) \hspace{2mm} \vert \hspace{2mm} y \text{ is almost positive }\} \vert\leq 2^n\vert \{R_y(A) \hspace{2mm} \vert \hspace{2mm} y \text{ is almost positive }\} \vert =$$
    $$=2^n\vert \{\Psi_y^{(A)} \hspace{2mm} \vert \hspace{2mm} y \text{ is almost positive }\} \vert\leq 2^{l+n} <\infty.$$
\end{proof}
\end{proposition}
Now we are ready to show that NCCC have the $K_0$-embedding Property.

\begin{proposition}\label{NCCC_K_0_embedding_property}
Let $A$ be a NCCC. Then $A$ has the $K_0$-embedding Property.
\begin{proof}
Let $A$ be a NCCC. For given $y=[p]_0-[q]_0\in K_0(A)$ almost positive, consider the following sequence of maps:
$$\begin{tikzcd}
A \arrow[r, "\Psi_y^{(A)}"] & R_y(A) \arrow[r, "\pi"] & B_y(A) \arrow[r, "\rho"] & D_y(A)
\end{tikzcd}$$
where $\pi:R_y(A)=F_1^{y,A}\oplus B_y(A)\rightarrow B_y(A)$ is the projection to the second coordinate, $\rho$ is the embedding from Corollary \ref{AF-embedding} with respect to (any) $\tau\in T(B_y(A))$ faithful and $D_y(A)$ is an AF algebra. Set
$$h_y^{(A)}:=\rho \circ \pi \circ \Psi_y^{(A)}:A\rightarrow D_y(A).$$

By Proposition \ref{main_technical_result} (Property iii),
\begin{equation}
    \label{6.9}
    \vert \{D_y(A) \hspace{2mm} \vert \hspace{2mm} y \text{ is almost positive }\} \vert = \vert \{h_y^{(A)} \hspace{2mm} \vert \hspace{2mm} y \text{ is almost positive }\} \vert <\infty.
\end{equation}
Denote
$$h_0^{(A)}:=\bigoplus_y h_y^{(A)} \text{ and } D(A)=\bigoplus_y D_y(A)$$
where the sum is over all $y$ that are almost positive. By (6.9) the sums can be taken to be finite if we never count the same summand twice, so we may assume that $D(A)$ is AF. Moreover, set
$$h^{(A)}=\sigma^{(A)}\oplus h_0^{(A)}:A\rightarrow F_0\oplus (\oplus_{i=1}^l C(D^{n_i}, M_{r_i}))\oplus D(A):=E(A).$$
We will show that  $h^{(A)}$ sends singular elements to singular elements. Indeed, suppose that
\begin{equation}
    \label{6.10}
    (h^{(A)})_*(y)>0
\end{equation}
for some $y\in K_0(A)$. Obviously, $y$ has to be almost positive. Moreover,
\begin{equation}
    \label{6.11}
    (h_y^{(A)})_*(y)\geq 0.
\end{equation}
Because of the construction of the embedding $\rho$ and (6.11),
it follows that $$y^{(2)}\notin \Inf K_0(B_y(A))\backslash \{\text{ torsion elements }\}.$$
where $y^{(2)}$ is as in statement of Proposition \ref{main_technical_result}.
But by Proposition \ref{main_technical_result} (Property $ii$),
$$y^{(2)}\in \Inf K_0(B_y(A)).$$
Hence, there is $N\in \N^*$ such that $Ny^{(2)}=0$. Thus
$$N(\Psi_y^{(A)})_*(y)\geq 0.$$
Finally, by Proposition \ref{main_technical_result} (Property $i$), there is $M>0$ such that $My\geq 0$. But, because of (6.10), it can't be $My=0$. Thus, it follows that $My>0$ as desired. \par
The result follows from the fact that $E(A)\in \mathcal{O}$ so it has the $K_0$-embedding Property by Proposition \ref{direct_sums}.
\end{proof}
\end{proposition}

Assume that there exists $y\in \ker((h^{(A)})_*)$ that is non-torsion and singular. Then $(\sigma_y^{(A)})_*(y)=0$, so $y$ is almost positive. Notice that we are in Case 2 of the proof of Proposition \ref{main_technical_result} so $\Psi_y^{(A)}=id$ and $y$ is a non-torsion infinitesimal. Thus $(h_y^{(A)})_*(y)=\rho_*(y)\neq 0$ because $\rho$ is the map from Corollary \ref{AF-embedding}. This contradicts $y\in \ker((h^{(A)})_*)$. It follows that $h$ sends non-torsion singular elements to non-torsion singular elements. \par 
Note that because the class of subhomogeneous $C^*$-algebras is closed under taking quotients, every ASH algebra can be written as an inductive limit of subhomogeneous $C^*$-algebras with injective connecting maps. So, if we combine Proposition \ref{unitization}, Proposition \ref{NCCC_K_0_embedding_property}, Proposition \ref{ASH_approximation}, Proposition \ref{local_approximations} and Remark \ref{inductive_limit}, we get the result we are aiming for:

 \begin{proposition}\label{ASH_K_0_embedding_property}
 Let $A$ be a separable ASH algebra. Then $A$ has the $K_0$-embedding property.
 \end{proposition}

Now Theorem \ref{main-theorem} can be deduced from all the aforementioned. \par
\vspace{10mm}
\textbf{Proof of Theorem \ref{main-theorem}:} \par
Let $A$ and $\mathcal{Y}$ as in the hypothesis. Because of Remark \ref{drop_compacts}, it is enough to show that $A$ has the $K_0$ embedding property. Because of Proposition \ref{UHF-stability} it is enough to show that every $A$ that can be locally approximated by algebras in $\mathcal{Y}$, has the $K_0$-embedding Property. Note that $\mathcal{Y}$ contains only separable, nuclear and quasidiagonal $C^*$-algebras, so by Proposition \ref{local_approximations} it is enough to show that all $C^*$-algebras in $\mathcal{Y}$ have the $K_0$ embedding property. But if $R$ is a NCCC, then by Proposition \ref{NCCC_K_0_embedding_property} and the observation right after its proof, there exists a faithful *-homomorphism $h:R\rightarrow E(R)$ such that $E(R)\in \mathcal{O}$ and $h_*$ sends non-torsion singular elements to non-torsion singular elements. So, Proposition \ref{local_approximations}, Proposition \ref{crossed_products} and Proposition \ref{direct_sums} yield that every $C^*$-algebra in $\mathcal{Y}$ has the $K_0$-embedding Property. \qed

\vspace{5mm}

\begin{appendix}
\section{}

For the sake of completion, we will prove the following proposition, which is mentioned (without proof) on \cite[p. 84]{blackadar1998k}. This proposition is essential for the proof of Proposition \ref{crossed_products}.
\begin{proposition}\label{appendix_main}
Let $A$ be a unital and separable $C^*$-algebra, $\sigma:\Z\rightarrow Aut(A)$ be an action and $\tau\in T(A)$ a $\sigma$-invariant trace. Then every trace extending $\tau$ induces the same state in $K_0(A\rtimes_{\sigma} \Z)$.
\end{proposition}

Let $E:A\rtimes_{\sigma} \Z \rightarrow A$ be the conditional expectation that sends $\sum_{g\in \Z}a_g g$ to $a_0$. For every $\sigma$-invariant trace $\tau\in T(A)$ , $\tau \circ E\in T(A\rtimes_{\sigma} \Z)$. Thus invariant traces can always be extended to traces in the crossed product, so the statement of the aforementioned proposition makes sense. \par 
For the definition of the functions $\Delta_{\tau}$ and $\underline{\Delta}_{\,\tau}$ that we will use throughout this appendix, we refer the reader to  \cite[p.378]{10.2307/24714069} and \cite[p. 379]{10.2307/24714069} respectively. Our starting point is the following known proposition.
\begin{proposition}[cf. Prop. 2, \cite{10.2307/24714069}] \label{appendix 1}
Let 
$$\begin{tikzcd}
0 \arrow[r] & J \arrow[r, "i"] & B \arrow[r, "\pi"] & A \arrow[r] & 0
\end{tikzcd}$$ be a short exact sequence of $C^*$-algebras and $\tau\in T(A)$.
Then
$$\begin{tikzcd}
0 \arrow[r] & \hat{\tau} \circ \pi_*(K_0(B)) \arrow[r] & \hat{\tau}(K_0(A)) \arrow[r, "q"] & \underline{\Delta}_{\,\tau}(\ker(i_*))    \arrow[r] & 0
\end{tikzcd}$$
where the first map is the inclusion of the two subgroups of $\R$ and $q$ is the restriction of the quotient map
$$q: \R \rightarrow \R\slash \hat{\tau} \circ \pi_*(K_0(B))$$
to  $\hat{\tau}(K_0(A))$, is a short exact sequence.
Moreover, for every $p\in P_{\infty}(A)$, $q(\tau(p))=\underline{\Delta}_{\,\tau}(-\delta_0[p])$, where $\delta_0$ is the boundary map in K-theory.

\end{proposition}
The first part of the statement is contained on the statement of \cite[Prop. 2]{10.2307/24714069}, while the second one is (explicitly) shown in its proof. \par

Because the result we want to show concerns K-theory of crossed products with the integers, we need to recall the Pimsner-Voiculescu 6-term exact sequence (\cite{10.2307/24713853}).
\begin{theorem} (Pimsner-Voiculescu 6-term exact sequence)\label{P-V}
Let $A$ be a unital $C^*$-algebra and $\sigma: \Z \rightarrow Aut(A)$ be an action. Then there exists a short exact sequence
\begin{equation}
    \label{ap_1}
     \begin{tikzcd}
0 \arrow[r] & A\otimes \K \arrow[r, "\phi"] & T_{\sigma} \arrow[r, "\psi"] & A\rtimes \Z \arrow[r] & 0
\end{tikzcd}
\end{equation}
for a $C^*$-algebra $T_{\sigma}$ (see \cite[2.1]{666} for its definition). Moreover, if $\iota: A\hookrightarrow A\rtimes_{\sigma} \Z$ and $j:A\hookrightarrow T_{\sigma}$ are the natural inclusions, they satisfy $\iota= \psi \circ j$. In addition $j$ induces isomorphisms on both $K_0$ and $K_1$. Furthermore, the following diagram
\begin{equation}
    \label{ap_2}
    \begin{tikzcd}
 K_i(A) \arrow[r, "1-\sigma_*"] \arrow[d, "\beta_*"] & K_i(A) \arrow[r, "\iota_*"] \arrow [d,"j_*"] & K_i(A\rtimes_{\sigma} \Z)  \arrow[d, "id"]  \\
K_i(A\otimes \K) \arrow[r, "\phi_*"] & K_i(T_{\sigma}) \arrow[r, "\psi_*"] & K_i(A\rtimes_{\sigma} \Z)
 \end{tikzcd}
\end{equation}
,where i=0,1 and $\beta: A\hookrightarrow A\otimes \K$ is the natural embedding which yields isomorphisms in K-theory by stability, is commutative. \par
Finally the short exact sequence (\ref{ap_1}) induces the following 6-term exact sequence in K-theory
$$\begin{tikzcd}
K_0(A)\arrow{r}{1-\sigma_*} & K_0(A) \arrow{r}{\iota_*} & K_0(A\rtimes_{\sigma}\Z) \arrow{dr}{\delta_0}\\
&K_1(A\rtimes_{\sigma}\Z)\arrow{ul}{\delta_1} & K_1(A) \arrow{l}{\iota_*} & K_1(A) \arrow {l}{1-\sigma_*}
\end{tikzcd}$$
\end{theorem}
\vspace{8mm}
In order to be precise, we need to mention that $\phi$ and $\psi$ are defined in \cite[Lemma 2.3]{666} and \cite[Lemma 2.4]{666} respectively,  (\ref{ap_1}) is \cite[Prop 2.7]{666} and (\ref{ap_2}) is deduced after combining  \cite[Prop 2.14]{666} with $\iota=\psi \circ j$. \par
Let now $\tau\in T(A)$ be a $\sigma$-invariant trace and $\tau_1\in T(A\rtimes_{\sigma} \Z)$ a trace extending $\tau$. Consider the map
$$\underline{\Delta}_{\,\tau}^{\sigma}: \ker(1-\sigma_*)\leq K_1(A) \rightarrow \R\slash \hat{\tau}(K_0(A))$$ via
$$\underline{\Delta}_{\,\tau}^{\sigma}([u]_1)=\Delta_{\tau}(u\sigma(u^{-1})).$$
Moreover, by commutativity of the diagram \ref{ap_2}, $\hat{\tau}_1 \circ \psi_*(K_0(T_{\sigma}))=\hat{\tau} \circ j^{-1}_*(K_0(T_{\sigma}))=\hat{\tau}(K_0(A))$. The latter holds because $j_*$ is an isomorphism. Furthermore, again by commutativity, $\beta_*(\ker(1-\sigma_*))=\ker(\phi_*)$. \par
By the proof of \cite[Theorem 3]{10.2307/24714069}, $\underline{\Delta}_{\,\tau}^{\sigma}=\underline{\Delta}_{\,\tau_1} \circ \beta_*$  (this is essentially what Pimsner shows on the proof; note slightly different notation).

So, by applying Proposition \ref{appendix 1} to the exact sequence (\ref{ap_1}), we get, for every $p\in P_{\infty}(A\rtimes_{\sigma} \Z)$,
\begin{equation}
\label{ap_3}
q(\tau_1(p))=\underline{\Delta}_{\,\tau_1}(\beta_*([u]_1)=\underline{\Delta}_{\,\tau}^{\sigma}([u]_1).
\end{equation}
for some $u\in U_{\infty}(K_1(A))$.
Notice that $q(\tau_1(p))$ is independent of the trace extension $\tau_1$. Moreover, we have the following exact sequence
\begin{equation}
\label{ap_4}
\begin{tikzcd}
0 \arrow[r] & \hat{\tau}(K_0(A)) \arrow[r] & \hat{\tau}_1(K_0(A\rtimes_{\sigma}\Z)) \arrow[r, "q"] & \underline{\Delta}_{\,\tau}^{\sigma}(\ker(1-\sigma_*))    \arrow[r] & 0
\end{tikzcd}
\end{equation}

Now we can prove the result we are aiming for: \par
\emph{Proof of Proposition \ref{appendix_main}:} Let $A$ be a separable and unital $C^*$-algebra, $\sigma: \Z \rightarrow Aut(A)$ an action, $\tau \in T(A)$ a $\sigma$-invariant trace.
For fixed $x\in K_0(A\rtimes_{\sigma} \Z)$, consider the set
$$L_x:=\{\hat{\tau}_1(x)-\hat{\tau}_2(x) \hspace{3mm} | \hspace{3mm} \tau_1,\tau_2 \text{ extend } \tau\}.$$
Of course $0\in L_x$. Assume that for some $x$, $L_x\neq \{0\}$. Then, there exist $\tau_1,\tau_2$ extending $\tau$ such that $\hat{\tau}_1(x)-\hat{\tau}_2(x) \neq 0.$ By considering convex combinations $w\tau_1+(1-w)\tau_2$, we can see that $L_x$ has to contain an interval around zero, so it has to be uncountable. On the other hand, by (\ref{ap_3}), $q(\hat{\tau}_3(x))=q(\hat{\tau}_4(x))$ for every $\tau_3,\tau_4 \in T(A\rtimes_{\sigma} \Z)$ that extend $\tau$. By the exactness of (\ref{ap_4}), $\hat{\tau}_3(x)-\hat{\tau}_4(x)\in \hat{\tau}(K_0(A))$. Thus $L_x\subset \hat{\tau}(K_0(A))$. But $A$ is separable, hence $\hat{\tau}(K_0(A))$ is countable, contradiction. Hence $L_x=\{0\}$ for every $x$. Proof is complete. \qed

\end{appendix}
\vspace{5mm}
\textbf{Acknowledgements:} I would like to thank my advisor Marius Dadarlat for introducing me to the problem and for making useful comments on earlier drafts of this paper. I would like also to thank Jose Carrion and Chris Schafhauser for pointing out to me how Theorem \ref{lifting_theorem} can be deduced from \cite[Cor. 5.4]{Schafhauser2018SubalgebrasOS}. Finally, I would like to thank the referee for a close reading of the paper and for useful comments.

\vspace{10mm}

\begin{center}

\bibliographystyle{abbrv}
\small

\bibliography{iason}
\end{center}
\emph{Iason Moutzouris: Department of Mathematics, Purdue University, West Lafayette, IN 47907, USA }\par

\emph{Email address: imoutzou@purdue.edu} \par

\end{document}